\documentclass[11pt]{article}

\usepackage{graphicx, subfigure}
\usepackage[top=1.05in,bottom=1.05in,left=1.05in,right=1.05in]{geometry}
\usepackage[numbers, square]{natbib}

\usepackage{amsmath,amsfonts,amscd,amssymb,bm,epsfig,epsf,bbm,constants,amssymb}
\usepackage[usenames]{color}

\usepackage[ruled, vlined, lined, commentsnumbered]{algorithm2e}

\definecolor{darkred}{RGB}{100,0,0}
\definecolor{darkgreen}{RGB}{0,100,0}
\definecolor{darkblue}{RGB}{0,0,150}
\definecolor{red}{RGB}{255,0,0}

\usepackage{hyperref}
\hypersetup{colorlinks=true, linkcolor=darkred, citecolor=darkgreen, urlcolor=darkblue}
\usepackage{url}

\usepackage{prettyref}
\newrefformat{eq}{(\ref{#1})}
\newrefformat{chap}{Chapter~\ref{#1}}
\newrefformat{sec}{Section~\ref{#1}}
\newrefformat{algo}{Algorithm~\ref{#1}}
\newrefformat{fig}{Fig.~\ref{#1}}
\newrefformat{tab}{Table~\ref{#1}}
\newrefformat{rmk}{Remark~\ref{#1}}
\newrefformat{clm}{Claim~\ref{#1}}
\newrefformat{def}{Definition~\ref{#1}}
\newrefformat{cor}{Corollary~\ref{#1}}
\newrefformat{lmm}{Lemma~\ref{#1}}
\newrefformat{lemma}{Lemma~\ref{#1}}
\newrefformat{prop}{Proposition~\ref{#1}}
\newrefformat{app}{Appendix~\ref{#1}}
\newrefformat{ex}{Example~\ref{#1}}
\newrefformat{exer}{Exercise~\ref{#1}}
\newrefformat{soln}{Solution~\ref{#1}}

\newcommand{\pth}[1]{\left( #1 \right)}

\newcommand{\sth}[1]{\left\{ #1 \right\}}
\newcommand{\wt}{\widetilde}
\newcommand{\wh}{\widehat}
\newcommand{\eps}{\epsilon}

\newtheorem{theorem}{Theorem}[section]
\newtheorem{lemma}[theorem]{Lemma}

\newcommand{\Var}{\textrm{Var}}

\renewcommand{\P}{\operatorname{\mathbb{P}}}
\newcommand{\E}{\operatorname{\mathbb{E}}}
\newcommand{\norm}[1]{{\left\lVert{#1}\right\rVert}}

\newcommand{\indep}{\rotatebox[origin=c]{90}{$\models$}}

\newcommand{\vct}[1]{\bm{#1}}
\newcommand{\mtx}[1]{\bm{#1}}

\newcommand{\supp}[1]{\operatorname{supp}(#1)}

\numberwithin{equation}{section}

\definecolor{eac}{RGB}{200,50,50}
\definecolor{XL}{RGB}{200,50,200}

\numberwithin{equation}{section}
\def \endprf{\hfill {\vrule height6pt width6pt depth0pt}\medskip}
\newenvironment{proof}{\noindent {\bf Proof} }{\endprf\par}

\begin{document}

\title{Optimal Rates of Convergence for Noisy Sparse Phase Retrieval via Thresholded Wirtinger Flow} 
\author{T.~Tony Cai, Xiaodong Li, and Zongming Ma
\\
~
\\
University of Pennsylvania}
\maketitle

\begin{abstract}
This paper considers the noisy sparse phase retrieval problem: recovering a sparse signal $\vct{x} \in \mathbb{R}^p$ from noisy quadratic measurements $y_j = (\vct{a}_j' \vct{x} )^2 + \epsilon_j$, $j=1, \ldots, m$, with independent sub-exponential noise $\epsilon_j$. The goals are to understand the effect of  the sparsity of $\vct{x}$ on the estimation precision and to construct a computationally feasible estimator to achieve the optimal rates. Inspired by the \emph{Wirtinger Flow} \cite{Candes14} proposed for noiseless and non-sparse phase retrieval, a novel thresholded gradient descent algorithm is proposed and it is shown to adaptively achieve the minimax optimal rates of convergence over a wide range of sparsity levels when the $\vct{a}_j$'s are independent standard Gaussian random vectors, provided that the sample size is sufficiently large compared to the sparsity of $\vct{x}$.
\end{abstract}

\section{Introduction}

In a range of fields in science and engineering, researchers face the problem of recovering a $p$-dimensional signal of interest $\vct{x}$ by probing the signal via a set of $p$-dimensional sensing vectors $\vct{a}_j$, $j=1,\dots,m$, and hence the observations are the $(\vct{a}_j'\vct{x})$'s contaminated with noise.
This gives rise to the linear regression model in statistical terminology where $\vct{x}$ is the regression coefficient vector and $\mtx{A} = [\vct{a}_1,\dots,\vct{a}_m]'$ is the design matrix.
There is an extensive literature on the theory and methods for the estimation/recovery of $\vct{x}$ under such a linear model. However, in many important applications, including X-ray crystallography, microscopy, astronomy, diffraction and array imaging, interferometry, and quantum information, it is sometimes impossible to observe $\vct{a}_j'\vct{x}$ directly and the measurements that one is able to obtain are the magnitude/energy of $\vct{a}_j' \vct{x}$ contaminated with noise.
In other words, the observations are generated by the following \emph{phase retrieval} model:
\begin{align}
\label{eq:measurement}
y_j = |\vct{a}_j' \vct{x} |^2 + \epsilon_j, \qquad j=1,\dots,m,
\end{align}
where $\vct{\epsilon}=(\epsilon_1, \ldots, \epsilon_m)'$ is a vector of stochastic noise with $\E \vct{\epsilon}= \vct{0}$. Note that $\E(y_j)=|\vct{a}_j' \vct{x} |^2$, so in the real case, \eqref{eq:measurement} can be treated as a generalized linear model with the multi-value link function $g(z):=\pm \sqrt{z}$. We refer interested readers to \cite{shechtman2014phase} and the reference therein for more detailed discussions on scientific and engineering background for this model. 

%

In many applications, especially those related to imaging, 
the signal $\vct{x} \in \mathbb{R}^p$ admits a sparse representation under some known and deterministic linear transformation. 
Without loss of generality, we assume in the rest of the paper that such a linear transform has already taken place and hence the signal $\vct{x}$ is sparse itself. 
In this case, model  \eqref{eq:measurement} is referred to as the \emph{sparse phase retrieval} model. 
In addition, we consider the case where $\vct{\epsilon}$ are independent centered sub-exponential random errors. 
This is motivated by the observation that in the application settings where model \eqref{eq:measurement} is appropriate, especially in optics, heavy-tailed noise may arise due to photon counting.

Efficient computational methods for phase retrieval have been proposed in the community of optics, and they are mostly based on the seminal work by Gerchberg, Saxton, and Fienup \cite{GS1972, Fienup1982}. The effectiveness of these methods relies on careful exploration of prior information of the signal in the spatial domain. Moreover, these methods were revealed later as non-convex successive projection algorithms \cite{LS1984, BCL2002}. This provides insight for occasional observation of stagnation of iterates and failure of convergence. 

Recently, inspired by multiple illumination, novel computational methods were proposed for phase retrieval without exploring and employing a priori information of the signal. These methods include semidefinite programming \cite{chai2011array, candes2013phase, candes2012phaselift, waldspurger2012phase, PRCDP}, polarization \cite{alexeev2012phase}, alternating minimization \cite{netrapalli2013phase}, gradient methods \cite{Candes14}, alternating projection \cite{MTW2014}, etc. More importantly, profound and remarkable theoretical guarantees for these methods have also been established. As for noiseless sparse phase retrieval, semidefinite programming has been proven to be effective with theoretical guarantees \cite{li2012sparse, OJFEH2015, jaganathan2012robust}. Other empirical methods for sparse phase retrieval include belief propagation \cite{SR2012} and greedy methods \cite{SBE2014}.

Regarding noisy phase retrieval, some stability results have been established in the literature; See \cite{candes2012solving, Soltanolkotabi2014, CC15}. 
In particular, stability results have been established in \cite{CCG2014} for noisy sparse phase retrieval by semidefinite programming, though the authors did not study the optimal dependence of the convergence rates on the sparsity of the signal and the sample size.
Nearly minimax convergence rates for sparse phase retrieval with Gaussian noise have been established in \cite{LM2013} under sub-gaussian design matrices. However, the optimal rates are achieved by empirical risk minimization under sparsity constraints, in which both the objective function and the constraint are non-convex, implying that the procedure is not computationally feasible.

In the present paper, we establish the minimax optimal rates of convergence for noisy sparse phase retrieval under sub-exponential noise, and propose a novel thresholded gradient descent method in order to estimate the signal $\vct{x}$ under the model \eqref{eq:measurement}. 
For conciseness, we focus on the case where the signal and the sensing vectors are all real-valued, and the key ideas extend naturally to the complex case.
The theoretical analysis sheds light on the effects of the sparsity of the signal $\vct{x}$  and the presence of sub-exponential noise on the minimax rates for the estimation of $\vct{x}$ under the $\ell_2$ loss, as long as the sensing vectors $\vct{a}_j$'s are independent standard Gaussian vectors. Combining the minimax upper and lower bounds given in Section 3, the optimal rate of convergence for estimating the signal $\vct{x}$ under the $\ell_2$ loss is ${\sigma\over \|\vct{x}\|_2}\sqrt{k\log p\over m}$, where $k$ is the sparsity of $\vct{x}$, $\|\cdot\|_2$ is the usual Euclidean norm, and $\sigma$ characterizes the noise level. Moreover,  it is shown that the thresholded gradient descent procedure is both rate-optimal and computationally efficient, and the sample size requirement matches the state-of-the-art result in computational sparse phase retrieval under structureless Gaussian design matrices.

%
We explain some notation used throughout the paper. For any $n$-dimensional vector $\vct{v}=(v_1, \ldots, v_n)'$ and a subset $S \subset \{1, \ldots, n\}$, we denote by $\vct{v}_S$ the $n$-dimensional vector by keeping the coordinates of $\vct{v}$ with indices in $S$ unchanged, while changing all other components to zero. We also denote $\|\vct{v}\|_q:=(v_1^q+\ldots+v_n^q)^{1/q}$ for $q \geq 1$, and $\|\vct{v}\|_\infty = \max_{1 \leq k \leq n}|v_k|$. Also denote $\|\vct{v}\|_0$ as the number of nonzero components of $\vct{v}$. For any matrix $\mtx{M} \in \mathbb{R}^{n_1 \times n_2}$, and any subsets $S_1 \in \{1, \ldots, n_1\}$ and $S_2 \in \{1, \ldots, n_2\}$, $\mtx{M}_{S_1 S_2} \in \mathbb{R}^{n_1 \times n_2}$ is defined by keeping the submatrix of $\mtx{M}$ with row index set $S_1$ and column index set $S_2$, while changing all other entries to zero. For any $q_1 \geq 1$ and $q_2 \geq 1$, we denote $\|\mtx{M}\|_{q_2 \rightarrow q_1}$ the induced norm from the Banach space $(\mathbb{R}^{n_2}, \|\cdot\|_{q_2})$ to $(\mathbb{R}^{n_1}, \|\cdot\|_{q_1})$. For simplicity, denote $\|\mtx{M}\|:= \|\mtx{M}\|_{2 \rightarrow 2}$. We also denote by $\mtx{I}_n$ the $n \times n$ identity matrix.

The rest of the paper is organized as follows: In Section \ref{sec:method}, we introduce in detail the thresholded gradient descent procedure, which consists of two steps. The first is an initialization step by applying a diagonal thresholding method to a matrix constructed with available data. The second step applies iterative thresholding procedure for the recovery of the sparse vector $\vct{x}$. Section \ref{sec:theory} establishes the minimax optimal rates of convergence for noisy sparse phase retrieval under the $\ell_2$ loss. The results show that the proposed thresholded gradient descent method is rate-optimal. In Section \ref{sec:simulation}, numerical simulations illustrate the effectiveness of thresholding in denoising, and demonstrate how the relative estimation error depends on the thresholding parameter $\beta$, sample size $m$, sparsity $k$, and the noise-to-signal ratio $\sigma/\|\vct{x}\|_2^2$.  In Section \ref{sec:discussion}, we discuss the connections between our thresholded gradient method for noisy sparse phase retrieval and related methods proposed in the literature for high-dimensional regression. The proofs are given in Section \ref{sec:proof} with some technical details deferred to the appendix.


\section{Methodology}
\label{sec:method}

The major component of the our method is a thresholded gradient descent algorithm to obtain a sparse solution to a given non-convex empirical risk minimization problem. Due to the non-convex nature of the problem, in order to avoid any local optimum that is far away from the truth, the initialization step is crucial. Thus, we also provide a candidate method which can be justified theoretically for yielding a good initializer. 
The methodology is proposed assuming that $\mtx{A}$ has standard Gaussian entries, though it could potentially also be used when such an assumption does not necessarily hold.

\subsection{Thresholded Wirtinger flow}

Given the sensing vectors $\vct{a}_j$ and the noisy magnitude measurements $y_j$ as in \eqref{eq:measurement} for $j=1,\dots, m$,
one can consider estimating $\vct{x}$ by minimizing the following empirical risk function
\begin{equation}
\label{eq:erf}
f(\vct{z}):= \frac{1}{4m} \sum_{j=1}^m \left(|\vct{a}_j' \vct{z}|^2 - y_j\right)^2.
\end{equation}

Statistically speaking, in the low-dimensional setup with fixed $p$ and $m \rightarrow \infty$, if the additive noises are heavy-tailed, least-absolute-deviations (LAD) methods might be more robust than least-squares methods. However, recent progress in modern linear regression analysis shows that least-squares could be preferable to LAD when $p$ and $m$ are proportional, even the noises are sub-exponential \cite{EBBLY2013}. Due to this surprising phenomenon, we simply take the least-squares empirical risk in \eqref{eq:erf}, although phase retrieval is a nonlinear regression problem, which could be very different from linear regression.  More importantly, close-form gradient methods can be induced from the empirical risk function in \eqref{eq:erf}, which is computationally convenient. To be specific, at any current value of $\vct{z}$, one updates the estimator by taking a step along the gradient direction
\begin{equation}
\label{eq:gradient}
\nabla f(\vct{z}) = \frac{1}{m} \sum_{j=1}^m \left(|\vct{a}_j'\vct{z}|^2 - y_j\right)(\vct{a}_j' \vct{z})\vct{a}_j
\end{equation}
until a stationary point is reached. Indeed, \citet{Candes14} showed that under appropriate conditions, initialized by an appropriate spectral method, a gradient method, referred to as Wirtinger flow, leads to accurate recovery of $\vct{x}$ up to a global phase in the complex domain and noiseless setting.

However, the direct application of gradient descent is not ideal for noisy sparse phase retrieval since it does not utilize the knowledge that the true signal $\vct{x}$ is sparse in order to mitigate the contamination of the noise. To incorporate this a priori knowledge, it makes sense to seek a ``sparse minimizer'' of \eqref{eq:erf}.
To this end, suppose we have a sparse initial guess $\vct{x}^{(0)}$ for $\vct{x}$.
To update $\vct{x}^{(0)}$ to another sparse vector, we may take a step along $\nabla f(\vct{x}^{(0)})$, and then sparsify the result by thresholding.

\begin{algorithm}[ht] 
\KwIn{Data $\sth{\vct{a_j}, y_j}_{j=1}^m$; initial estimator $\wh{\vct{x}}_0$; thresholding function $\mathcal{T}$; gradient tuning parameter $\mu$; thresholding tuning parameter $\beta$; number of iterations $T$.}

\KwOut{Final estimator $\wh{\vct{x}}$.}

\nl Initialize $n \leftarrow 0$ and $\wh{\vct{x}}^{(0)} = \wh{\vct{x}}_0$.


\Repeat{$n=T$}{
\nl Compute threshold level
\begin{equation}
\label{eq:threshold}
\tau(\widehat{\vct{x}}^{(n)})=
\sqrt{ \frac{\beta \log (mp)}{m^2} \sum_{j=1}^m \left(|\vct{a}_j'\widehat{\vct{x}}^{(n)}|^2 - y_j\right)^2 |\vct{a}_j'\widehat{\vct{x}}^{(n)}|^2}\, ;
\end{equation}

\nl Update 
\begin{equation}
\label{eq:iteration}
\widehat{\vct{x}}^{(n+1)} = \varphi(\widehat{\vct{x}}^{(n)}):= \mathcal{T}_{\frac{\mu}{\phi^2}\tau(\widehat{\vct{x}}^{(n)})}  \left(\widehat{\vct{x}}^{(n)} - \frac{\mu}{\phi^2} \nabla f (\widehat{\vct{x}}^{(n)})\right),
\end{equation}}

where $\nabla f$ is defined in \eqref{eq:gradient};

\nl Return $\wh{\vct{x}} = \wh{\vct{x}}^{(T)}$.
\caption{Thresholded Wirtinger flow for noisy sparse phase retrieval\label{algo:twf}}
\end{algorithm}

Indeed, if we were given the oracle knowledge of the support $S$ of $\vct{x}$, then we can reduce the problem to recovering $\vct{x}_S$ based on the $\{y_j, a_{jS}\}_{j=1}^m$. 
By avoiding estimating any coordinate of $\vct{x}$ in $S^c$, we could greatly reduce variance of the resulting estimator of $\vct{x}$.
In reality, we do not have such oracle knowledge and the additional thresholding step added on top of gradient descent is intended to mimic the oracle behavior by hopefully restricting all the updated coordinates on $S$.

Let $\mathcal{T}_\tau$ be any thresholding function satisfying 
\begin{align}
	\label{eq:thr}
\mathcal{T}_\tau(x) = 0,~~\forall x \in [-\tau, \tau], \quad\mbox{and}\quad
|\mathcal{T}_\tau(x) - x|\leq \tau, ~~\forall x\in \mathbb{R}.
\end{align}
For any vector $\vct{b} = (b_1,\dots, b_p)'$, let $\mathcal{T}_\tau(\vct{b}) = (\mathcal{T}_\tau(b_1),\dots, \mathcal{T}_\tau(b_p))'$.
With the foregoing definition, the proposed thresholded gradient descent method can be summarized as \prettyref{algo:twf}.
In view of the Wirtinger flow method for noiseless phase retrieval \cite{Candes14}, we name our approach the ``Thresholded Wirtinger Flow'' method.
The data-driven choice of the threshold level in \eqref{eq:threshold} is motivated by the following intuition.
Recall that we assume the sensing vectors $\{\vct{a}_j:j=1, \ldots, m\}$ are independent standard Gaussian vectors. For a fixed $\vct{z}$, if we act as if each $(|\vct{a}_j'\vct{z}|^2 - y_j)(\vct{a}_j'\vct{z})$ is a fixed constant, then the gradient in \eqref{eq:gradient} is a linear combination of Gaussian vectors and hence has i.i.d.~Gaussian entries with mean zero and variance $\frac{1}{m^2}\sum_{j=1}^m (|\vct{a}_j'\vct{z}|^2 - y_j)^2(\vct{a}_j'\vct{z})^2$.
Therefore, the threshold $\tau(\vct{z})$ is simply $\sqrt{\beta \log(mp)}$ times the standard deviation of these Gaussian random variables, which is essentially the universal thresholding in the Gaussian sequence model literature \cite{johnstone11}.
Although the above intuition is not exactly true, the resulting thresholds in \eqref{eq:threshold} are indeed the right choices as justified later in \prettyref{sec:theory}, and illustrated in \prettyref{sec:simulation}.
Notice that there are two tuning parameters $\mu$ and $\beta$, which should be treated as absolute constants. We will validate some theoretical choices and also provide practical choices later.




\subsection{Initialization}

\begin{algorithm}[ht]
\KwIn{Data $\sth{\vct{a_j}, y_j}_{j=1}^m$; tuning parameter $\alpha$.}

\KwOut{Initial estimator $\wh{\vct{x}}_0$.}

\nl Compute
\begin{equation}
\label{eq:phi}
\phi^2=\frac{1}{m} \sum_{j=1}^m y_j,
\end{equation}
and 
\begin{equation}
\label{eq:testing_seq}
I_l=\frac{1}{m} \sum_{j=1}^m y_j a_{jl}^2,~~ l=1, \ldots, p.
\end{equation}

\nl Select a set of coordinates 
\begin{equation}
\label{eq:initial_support}
\wh{S}_0 =\left\{ l \in [p]: I_l > \left(1+  \alpha \sqrt{\frac{\log (mp)}{m}}\right)\phi^2\right\}.
\end{equation}

\nl Compute a $p\times p$ matrix
\begin{equation}
	\label{eq:W-0}
\mtx{W}_{\wh{S}_0\wh{S}_0}:=\frac{1}{m} \sum_{j=1}^m y_j\vct{a}_{j \widehat{S}_0} \vct{a}_{j \widehat{S}_0}'.
\end{equation}

\nl Return 
\begin{align}
	\label{eq:x-init}
\wh{\vct{x}}_0 = \phi\, \wh{\vct{v}}_1
\end{align}
where $\widehat{\vct{v}}_1$ as the leading eigenvector of $\mtx{W}_{\wh{S}_0\wh{S}_0}$.
\caption{Initialization for \prettyref{algo:twf}\label{algo:init}}
\end{algorithm}

It is worth noting that the success of \prettyref{algo:twf} depends crucially on the initial estimator for two reasons.
First, the empirical risk \eqref{eq:erf} is a non-convex function of $\vct{z}$ and hence it could have multiple local minimizers. Hence the success of a gradient descent based approach depends naturally on the starting point.
Moreover, an accurate initializer can reduce the required number of iterations in the thresholded Wirtinger flow algorithm. In view of its crucial rule, we propose in \prettyref{algo:init} an initialization method which can be proven to yield a decent starting point for \prettyref{algo:twf} under our modeling assumption.

The motivation of the algorithm is similar to that of diagonal thresholding \cite{JohnstoneLu09} for sparse PCA: we want to identify a small collection of coordinates with big marginal signals and then compute an estimator of $\vct{x}$ by focusing only on these coordinates.
In particular, the quantity $I_l$ in \eqref{eq:testing_seq} captures the marginal signal strength of the $l$-th coordinate and $\wh{S}_0$ \eqref{eq:initial_support} selects all coordinates with big marginal signals.
Last but not least, \eqref{eq:W-0} and \eqref{eq:x-init} computes the initial estimator by focusing only on the coordinates in $\wh{S}_0$.
There is a tuning parameter $\alpha$ needed as input of the algorithm, which can be treated as an absolute constant. 
We will provide some justified theoretical choice later. 



%

\section{Theory}
\label{sec:theory}

We first establish the statistical convergence rate for the thresholded Wirtinger flow method under the case of ``Gaussian design", i.e., $\vct{a}_j \stackrel{iid}{\sim} \mathcal{N}(\vct{0}, \mtx{I}_p)$ for $j= 1, \ldots, m$ in \eqref{eq:measurement}.  
Moreover, we assume the signal $\vct{x}$ is $k$-sparse, i.e., $\|\vct{x}\|_0=k$, and the noises $\epsilon_1, \ldots, \epsilon_m$ are $m$ independent centered sub-exponential random variables with maximum sub-exponential norm $\sigma$, i.e., $\sigma := \max_{1 \leq i \leq m} \| \epsilon_i \|_{\psi_1}$. Here for any random variable $X$, its sub-exponential norm is defined as $\|X \|_{\psi_1}:=\sup_{p \geq 1} p^{-1} (\E |X|^p)^{\frac{1}{p}}$. This definition, as well as some fundamental properties of sub-exponential variables (such as Bernstein inequality), can be found in Section 5.2.4 of \cite{vershyninNARMT}.

\begin{theorem}
\label{thm:upper}
Suppose $\beta = 4$ in \eqref{eq:threshold}, and $\alpha=K\left(1+\frac{\sigma}{\|\vct{x}\|_2^2}\right)$ in \eqref{eq:initial_support} for some absolute constant $K$. Suppose $\mu \leq \mu_0$ in \eqref{eq:iteration} and $m\geq C \left(1+\frac{\sigma^2}{\|\vct{x}\|_2^4}\right)k^2 \log (mp)$. For all $t=1, 2, 3, \ldots$, there holds
\[
\sup_{\|\vct{x}\|_0 = k} \P_{(\mtx{A}, \vct{y}|\vct{x})} \left({\min\limits_{i = 0, 1} \|\widehat{\vct{x}}^{(t)} - (-1)^i\vct{x}\|_2} > \frac{1}{6}\left(1 - \frac{\mu}{16}\right)^t {\|\vct{x}\|_2} + C_0\frac{\sigma}{\|\vct{x}\|_2}\sqrt{\frac{k \log p}{m}}\right)\leq \frac{46}{m} + \frac{10}{e^{k}} + \frac{t}{mp^2}\]
where $\mu_0$, $C$, and $C_0$ are some absolute constants.
\end{theorem}

The proof is given in Section \ref{sec:proof}. Lemma \ref{thm:initialization} guarantees the efficacy of the initialization step \prettyref{algo:init}, and Lemmas \ref{thm:contraction} and \ref{lmm:induction} explain why the thresholded Wirtinger flow method leads to accurate estimation. Here $\beta = 4$ and $\alpha=K\left(1+\frac{\sigma}{\|\vct{x}\|_2^2}\right)$ are chosen for analytical convenience. The discussion of empirical choices of $\beta$, $\alpha$, and $\mu$ are deferred to Section \ref{sec:simulation}.

Let us interpret \prettyref{thm:upper} by considering the following cases. In the noiseless case, with high probability, we obtain $\min\limits_{i = 0, 1} \|\widehat{\vct{x}}^{(t)} - (-1)^i\vct{x}\|_2 \leq \frac{1}{6}\left(1 - \frac{\mu}{16}\right)^t \|\vct{x}\|_2$. This implies that thresholded gradient descent method leads to linear convergence to the original signal up to a global sign. 

In the noisy case, if $\mu>0$ is an absolute constant, by letting $t \asymp \log \left(1/\delta \right)$ where $\delta = \frac{\sigma}{\|\vct{x}\|_2^2}\sqrt{\frac{k \log p}{m}}$, we obtain $\min\limits_{i = 0, 1} \|\widehat{\vct{x}}^{(t)} - (-1)^i\vct{x}\|_2 \precsim\frac{\sigma}{\|\vct{x}\|_2}\sqrt{\frac{k \log p}{m}}$ with high probability. If the knowledge of $\delta$ is not available, by choosing $t=O(\log p)$, we can obtain $\min\limits_{i = 0, 1} \|\widehat{\vct{x}}^{(t)} - (-1)^i\vct{x}\|_2 \precsim\frac{\sigma}{\|\vct{x}\|_2}\sqrt{\frac{k \log p}{m}}+\frac{1}{p^c}$ for any predetermined $c>0$. The convergence rate $\frac{\sigma}{\|\vct{x}\|_2}\sqrt{\frac{k \log p}{m}}$ is better than the upper bound result established in \cite{LM2013}, which is achieved by the intractable sparsity constrained empirical risk minimization. Our contribution is to show that this rate can be obtained tractably by a fast algorithm.

Ignoring any polylog factor, the above convenient properties of thresholded Wirtinger flow are guaranteed by the sample size condition $m\gtrsim k^2$. When $m \ll p$, this condition is crucial for the effectiveness of initialization \prettyref{algo:init}. An immediate question is whether such a minimum sample size condition is in some sense necessary for any computationally efficient algorithm, if the sensing matrix is random and structureless? A similar phenomenon has been previously observed in the related but different problem of sparse principal component analysis. Assuming the hardness of the planted clique problem \cite{Alon98}, a series of papers \cite{Berthet13,Wang14,Gao14} have shown that a comparable minimum sample size condition is necessary for any estimator computable in polynomial time complexity to achieve consistency and optimal convergence rates uniformly over a parameter space of interest. In particular, it was shown in \cite{Gao14} that this is the case even for the most restrictive parameter space in sparse principal component analysis -- (discretized) Gaussian single spiked model with a sparse leading eigenvector. Establishing comparable computational lower bounds for sparse phase retrieval, especially under the Gaussian design, is an interesting project for future research.
 
 In the case when $m \gtrsim p$ ignoring any log factor, it is well-known that a consistent initializer can be obtained by spectral methods \cite{netrapalli2013phase, Candes14}, no matter whether $\vct{x}$ is sparse or not. In other words, the diagonal thresholding idea in \prettyref{algo:init} is not as crucial as in the case $m \ll p$. It is interesting to investigate whether $m\gtrsim k^2$ can be relaxed such that the optimal converge rates can still be achieved by thresholded Wirtinger flow.\\

The convergence rate $\frac{\sigma}{\|\vct{x}\|_2}\sqrt{\frac{k \log p}{m}}$ is essentially optimal. The following lower bound result has been essentially proven in \cite{LM2013}:

\begin{theorem} (\cite{LM2013})
	\label{thm:lower}
Let $\Theta(k,p,R) = \{\vct{x}\in \mathbb{R}^p: \|\vct{x} \|_2 =  R, \|\vct{x}\|_0 = k \}$. Suppose the $\vct{a}_j$'s are i.i.d.~$\mathcal{N}(0, \mtx{I}_p)$, the $\epsilon_j$'s are i.i.d.~$\mathcal{N}(0,\sigma^2)$, and they are mutually independent. There holds under model \eqref{eq:measurement}, 
\begin{align*}
\inf_{\wh{\vct{x}}} \sup_{\vct{x}\in \Theta(k,p,R)} \mathbb{P}_{(\vct{A}, \vct{y}|\vct{x})} \left(
\min_{i=0,1}\|\wh{\vct{x}} - (-1)^i \vct{x} \|_2
\geq C_0 \frac{\sigma}{R} \sqrt{\frac{k\log(ep/k)}{m}}\right) \geq \frac{1}{5},
\end{align*}
provided $m \geq C\left(\frac{\sigma^2}{\|\vct{x}\|_2^4}+1\right) k\log (ep/k)$, where both $C$ and $C_0$ are some absolute constants.
\end{theorem}

Notice that for a standard Gaussian variable with variance $\sigma^2$, its sub-exponential norm is a constant multiple of $\sigma$. 
For brevity, we do not scale the Gaussian noises such that their sub-exponential norms are strictly less than or equal to $\sigma$.

\section{Numerical Simulation}
\label{sec:simulation}

In this section, we report numerical simulation results to demonstrate how the relative estimation error depends on the thresholding parameter $\beta$, the noise-to-signal ratio (NSR) $\sigma/\|\vct{x}\|_2^2$, the sample size $m$, and the sparsity $k$. 
To guarantee fair comparison, we always fix the length of the signal $p=1000$ and the initialization parameter $\alpha=0.1$ (except for the first case on thresholding effect). Moreover, in each numerical experiment, we conservatively choose gradient parameter $\mu=0.01$, and the number of iterations $T=1000$ for thresholded Wirtinger flow. The resulting estimator is denoted as $\widehat{\vct{x}}=\widehat{\vct{x}}^{(1000)}$. With each fixed $k$, the support of $\vct{x}$ is uniformly distributed at random. The nonzero entries of $\vct{x}$ are i.i.d. $\sim \mathcal{N}(0, 1)$. The noise $\epsilon \sim \mathcal{N}(\vct{0}, \sigma^2 \mtx{I}_m)$, where $\sigma$ is determined by $\|\vct{x}\|_2$ and the choice of NSR $\sigma/\|\vct{x}\|_2^2$. As discussed before, the design matrix $\mtx{A}$ consists of independent standard Gaussian random variables.

\begin{enumerate}
\item

Thresholding effect: Fix $\alpha=0.1$, $m=7000$, $k=100$, and $\sigma/\|\vct{x}\|_2^2=1$. For each $\beta = 0, 0.25, 0.5, \ldots, 3$, we implement the algorithm for $10$ times with independently generated $\mtx{A}$, $\vct{x}$, and $\vct{\epsilon}$. and then take the average of the $10$ independent relative errors $\min(\|\widehat{\vct{x}} - \vct{x}\|_2, \|\widehat{\vct{x}} + \vct{x}\|_2)/\|\vct{x}\|_2$. The relation between the average relative error and the choice of $\beta$ is plotted as the red curve in Figure \ref{fig:Thresholding effect}. The result shows that the average relative error essentially decreases from $0.2365$ to $0.1151$ as the thresholding parameter increases from $0$ to $0.75$, and then increases slowly up to $0.1684$ as $\beta$ continues to increase to $3$.

We implement the above experiments again with the only difference $\alpha=0.5$. The relation curve between the relative estimation error and $\beta$ is plotted as the blue curve in \prettyref{fig:Thresholding effect}. It is clear that the performance of the algorithm is very close to the case $\alpha=0.1$.

\begin{figure}[ht]
\centering
 \includegraphics[width=80mm]{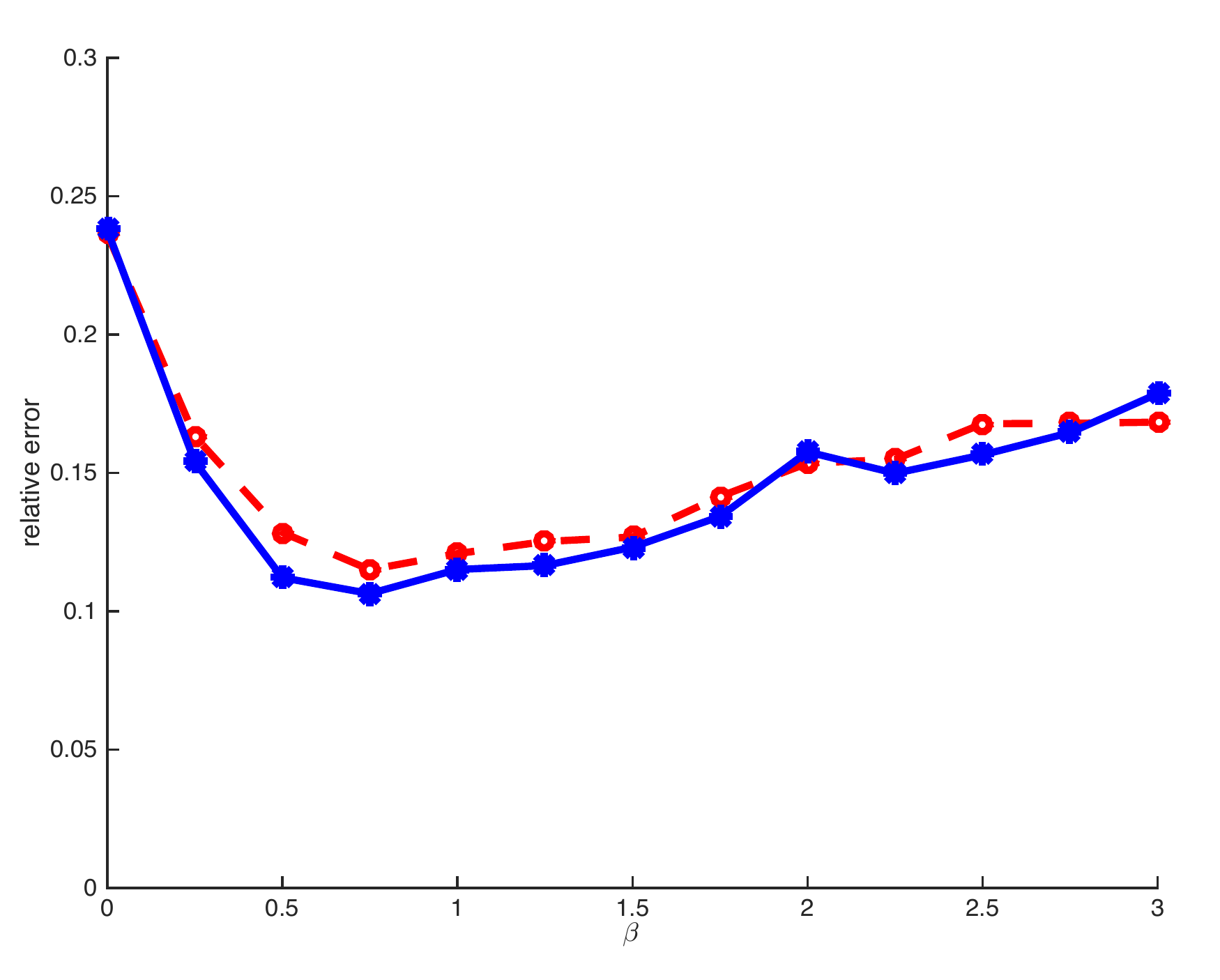}
\caption{The relation between the average relative error and the thresholding parameter $\beta$. Setup of parameters: $p=1000$, $m=1000$, $k=100$, $\sigma/\|\vct{x}\|_2^2=1$, $\mu=0.01$, and $T=1000$. Red curve with $\alpha=0.1$, while blue curve with $\alpha=0.5$.}
\label{fig:Thresholding effect}
\end{figure}

\item

Noise effect: Fix $m=7000$, $k=100$, and $\beta =1$. In each choice of NSR $\sigma/\|\vct{x}\|_2^2 = 0, 0.1, \ldots, 1$, with $5$ instances of $(\mtx{A}, \vct{x}, \vct{\epsilon})$ generated independently, we take the average of the relative error $\min(\|\widehat{\vct{x}} - \vct{x}\|_2, \|\widehat{\vct{x}} + \vct{x}\|_2)/\|\vct{x}\|_2$. In Figure \ref{fig:Noise effect}, it shows how the average relative error depends on NSR. The average relative error strictly increases from $0.0000$ to $0.1219$, as the NSR increases from $0$ to $1$.

\begin{figure}[h!]
\centering
\includegraphics[width=80mm]{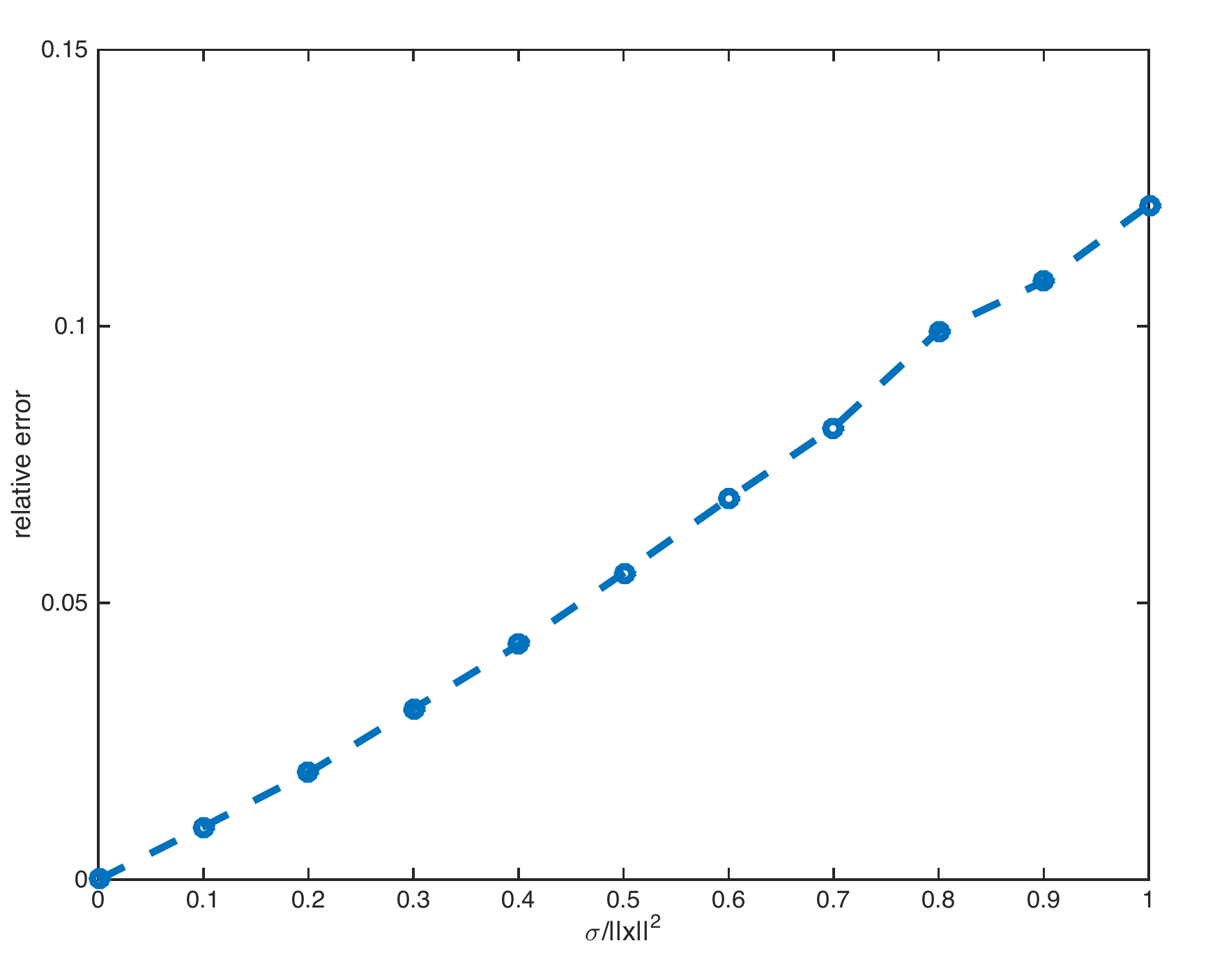}
\caption{The relation between the average relative error and the noise-to-signal-ratio $\sigma/\|\vct{x}\|_2^2$. Setup of parameters: $p=1000$, $m=1000$, $k=100$, $\beta=1$, $\alpha=0.1$, $\mu=0.01$, and $T=1000$.}
\label{fig:Noise effect}
\end{figure}

\item 

Sample size effect: Fix $k=100$, $\sigma/\|\vct{x}\|_2^2=1$, and $\beta =1$. In each choice of $m=2000, 3000, \ldots, 11000$, with $5$ instances of $(\mtx{A}, \vct{x}, \vct{\epsilon})$ generated independently, we take the average of the relative error $\min(\|\widehat{\vct{x}} - \vct{x}\|_2, \|\widehat{\vct{x}} + \vct{x}\|_2)/\|\vct{x}\|_2$. In Figure \ref{fig:Sample size effect}, it shows how the average relative error depends on the sample size. When the sample sizes are $2000$ and $3000$, i.e., twice and three times as large as $p$, the average relative errors are $0.8444$ and $0.3651$ respectively. In these cases, the thresholded gradient descent method leads to poor recovery of the original signal. When the sample size increases from $4000$ to $11000$, the average relative error decreases steadily from $0.1692$ to $0.0956$.

\begin{figure}[h!]
\centering
\includegraphics[width=80mm]{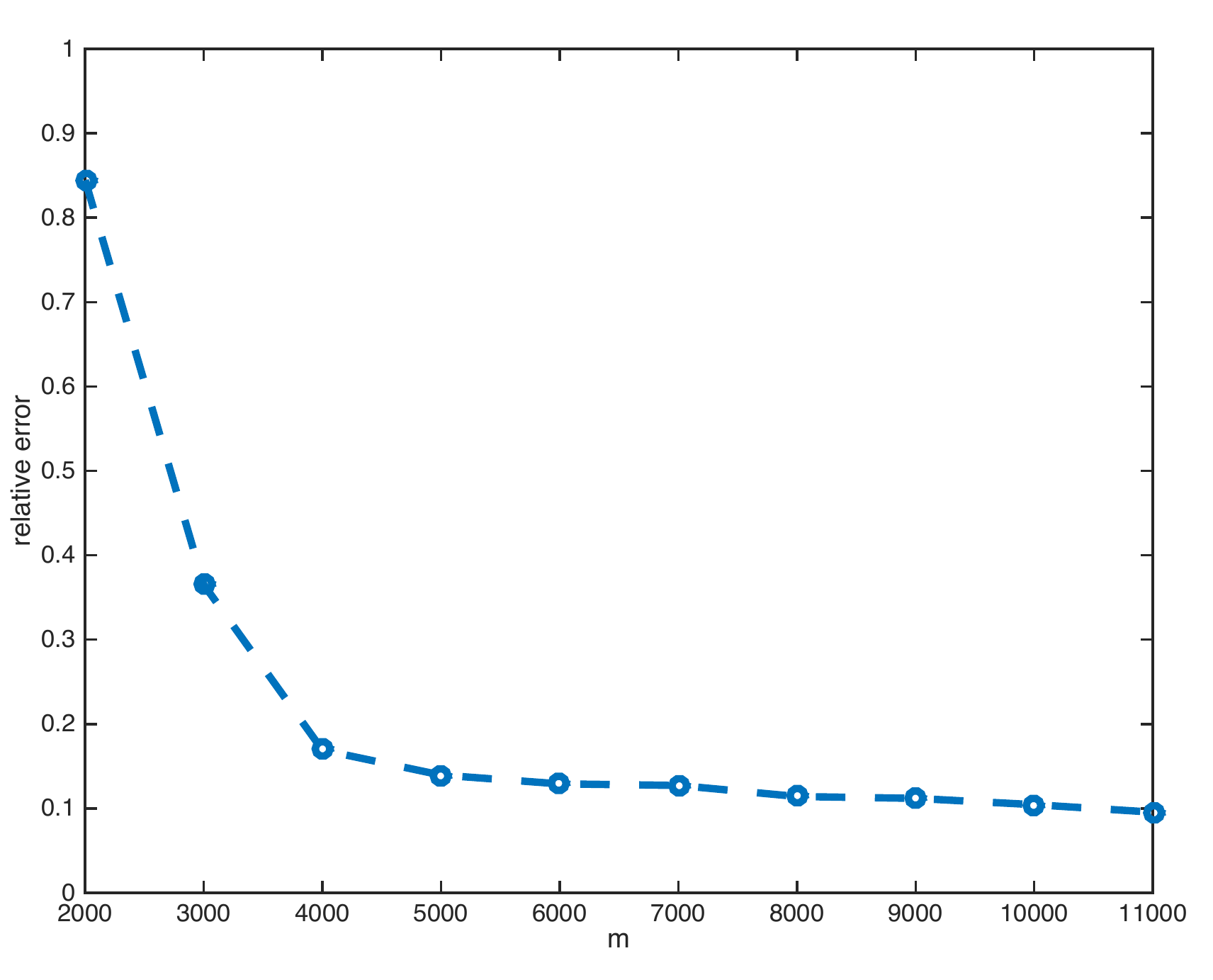}
\caption{The relation between the average relative error and the sample size $m$. Setup of parameters: $p=1000$, $\sigma/\|\vct{x}\|_2^2=1$, $k=100$, $\beta=1$, $\alpha=0.1$, $\mu=0.01$, and $T=1000$.}
\label{fig:Sample size effect}
\end{figure}

\item

Sparsity effect: Fix $m=7000$, $\sigma/\|\vct{x}\|_2^2=1$, and $\beta =1$. In each choice of sparsity $k = 25, 50, \ldots, 200$, with $10$ instances of $(\mtx{A}, \vct{x}, \vct{\epsilon})$ generated independently, we take the average of the relative error $\min(\|\widehat{\vct{x}} - \vct{x}\|_2, \|\widehat{\vct{x}} + \vct{x}\|_2)/\|\vct{x}\|_2$. Figure \ref{fig:Sparsity effect} demonstrates the relation between the average relative error and the sparsity. The average relative error essentially increases from $0.1059$ to $0.1666$, as the sparsity increases from $25$ to $200$.

\begin{figure}[h!]
\centering
\includegraphics[width=80mm]{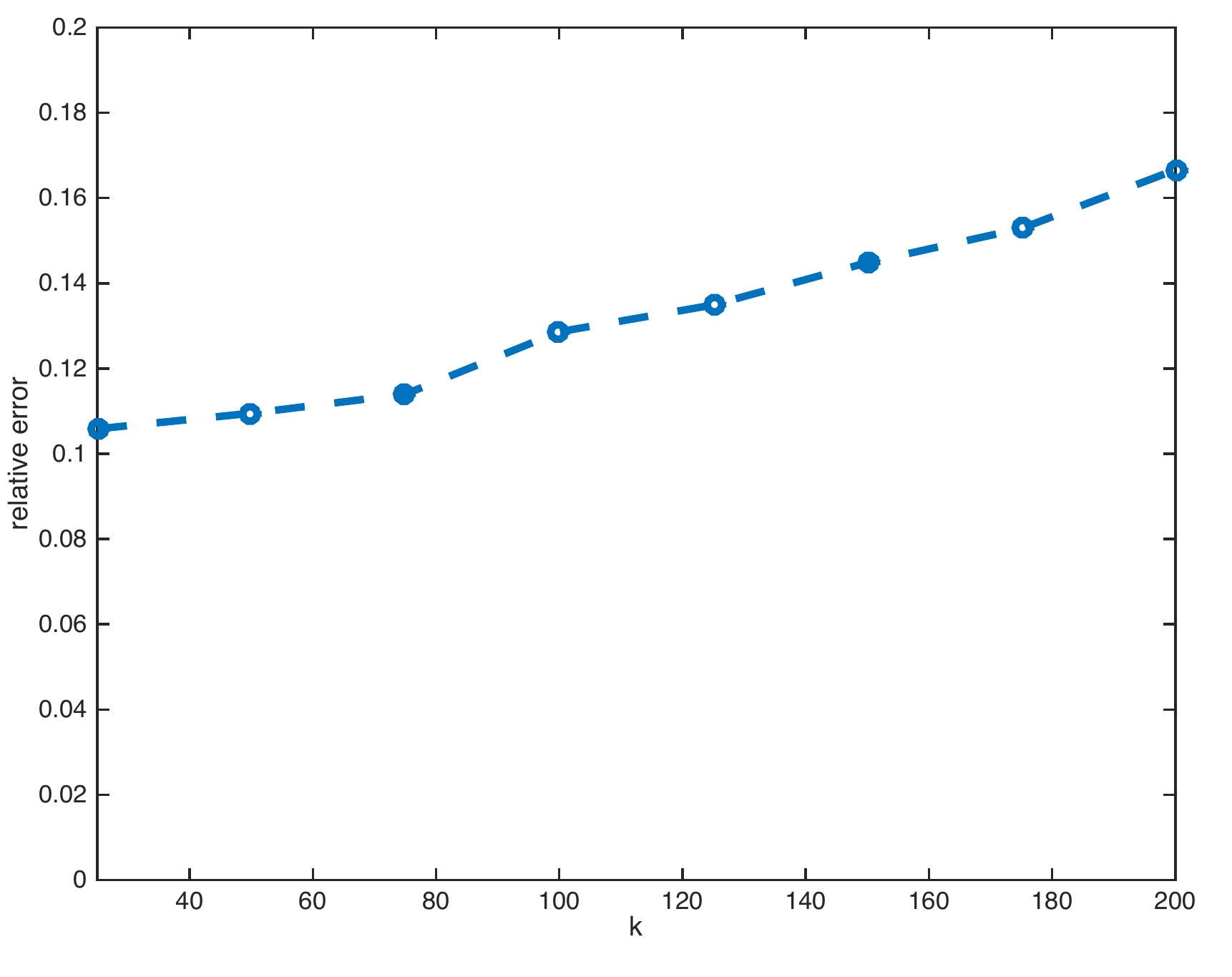}
\caption{The relation between the average relative error and the sparsity $k$. Setup of parameters: $p=1000$, $\sigma/\|\vct{x}\|_2^2=1$, $m=7000$, $\beta=1$, $\alpha=0.1$, $\mu=0.01$, and $T=1000$.}
\label{fig:Sparsity effect}
\end{figure}


\end{enumerate}

\section{Discussion}
\label{sec:discussion}

In this paper, we established the optimal rates of convergence for noisy sparse phase retrieval under the Gaussian design in the presence of sub-exponential noise, provided that the sample size is sufficiently large. Furthermore,  a thresholded gradient descent method called ``Thresholded Wirtinger Flow'' was introduced and shown to achieve the optimal rates. 

Iterative thresholding has been employed in a variety of problems in high-dimensional statistics, machine learning, and signal processing, under the assumption that the signal or parameter vector/matrix satisfies a sparse or low-rank constraint. Examples include compressed sensing/sparse approximation \cite{DDD2004, NT2009, MD2009, BD2009}, sparse principal component analysis \cite{Ma13, YZ2013}, high-dimensional regression \cite{ANW2012, YLZ2014, JTK2014}, and low-rank recovery \cite{CCS2010, KMO2010, lee2013near}.


Regarding the application of iterative thresholding and projected gradient methods in high-dimensional $M$-estimation, their statistical optimality has been established when the empirical risk function satisfies certain properties, such as restrictive strong convexity and smoothness (RSC and RSS) \cite{ANW2012, YLZ2014, JTK2014}. Although our thresholded gradient method aims to solve \eqref{eq:erf} for a sparse solution, the existing analytical framework for high-dimensional $M$-estimation does not apply to the sparse phase retrieval problem, since the empirical risk function in \eqref{eq:erf} does not satisfy RSC in general, no matter how large the sample size is. Instead, we have shown that thresholded gradient methods can achieve optimal statistical precision for signal recovery, even when the empirical risk function does not satisfy the common assumption of RSC.    

Besides thresholded gradient methods, convexly and non-convexly regularized methods are also widely-used for high-dimensional $M$-estimation. In fact, some iterative thresholding methods are induced by regularizations; See, e.g., \cite{DDD2004}. Therefore, an alternative candidate method for solving the noisy sparse phase retrieval problem is to penalize the empirical risk function in \eqref{eq:erf} before taking the minimum, in order to promote a sparse solution. The major difficulty is apparently the non-convexity of the empirical risk function. An interesting result in \cite{LW2013} guarantees the statistical precision of all local optima, as long as the non-convex penalty satisfies certain regularity conditions, and the empirical risk function, possibly non-convex, satisfies the restricted strong convexity.  A similar result appeared in \cite{WLZ2014}, in which the empirical risk function is required to satisfy a sparse eigenvalue (SE) condition. 
However, back to noisy sparse phase retrieval, the empirical risk function in \eqref{eq:erf} satisfies neither RSC nor SE in general, so there is no guarantee that all local optima are consistent. A natural question is whether some penalized version of \eqref{eq:erf} is strongly convex in a sufficiently large neighborhood of its global minimum, such that a tractable initializer lies in this neighborhood provided the sample size is sufficiently large. Another interesting question is whether the global minimizer of such penalized version of \eqref{eq:erf} is a rate-optimal estimator of the original sparse signal. We leave these questions for future research.


\section{Proof of \prettyref{thm:upper}}
\label{sec:proof}

In model \eqref{eq:measurement}, denote $S = \supp{\vct{x}}$, which implies $|S| = k$. Without loss of generality, we assume $S=\{1, \ldots, k\}$. As to the Gaussian design matrix $\mtx{A} \in \mathbb{R}^{m \times p}$, denote 
\begin{equation}
\label{eq:A_S}
\mtx{A}_{S}:=\begin{bmatrix} {\vct{a}_1}_{S}' \\ \vdots \\ {\vct{a}_m}_{S}' \end{bmatrix}, \quad \mtx{A}_{S^c}:=\begin{bmatrix} {\vct{a}_1}_{S^c}' \\ \vdots \\ {\vct{a}_m}_{S^c}' \end{bmatrix},
\end{equation}
both of which are in $\mathbb{R}^{m \times p}$.

%

For any two two random variables/vectors/matrices/sets $X$ and $Y$, we denote by $X \indep Y$ if $X$ and $Y$ are independent.

\begin{lemma}
\label{lmm:independence}
From the model \eqref{eq:measurement}, we have $\vct{y} \indep \mtx{A}_{S^c}$. Moreover, we have $\{I_1, \ldots, I_k\} \indep \mtx{A}_{S^c}$ and $\phi \indep \mtx{A}_{S^c}$, where $\phi$ and $\{I_1, \ldots, I_k\}$ are defined in \eqref{eq:phi} and \eqref{eq:testing_seq}, respectively.
\end{lemma}

\begin{proof}
The fact $\vct{y}=|\mtx{A}\vct{x}|^2+\vct{\epsilon} = |\mtx{A}_S\vct{x}_S|^2+\vct{\epsilon}$ implies straightforwardly that $\vct{y} \indep \mtx{A}_{S^c}$. By \eqref{eq:testing_seq}, we know for all $l=1, \ldots, k$, $I_l$ are defined by $\vct{y}$ and $\mtx{A}_S$, which implies that $I_l \indep \mtx{A}_{S^c}$ for all $l=1, \ldots, k$. Finally, by \eqref{eq:phi}, we know $\phi$ is determined uniquely by $\vct{y}$, which implies that $\phi \indep \mtx{A}_{S^c}$.
\end{proof}

%
%
%
%

\begin{lemma}
\label{lmm:norm_estimation}
On an event $\widetilde{E}_0$ with probability at least $1- \frac{3}{m}$, 
\[
1 - \left(2+C_0\frac{\sigma}{\|\vct{x}\|_2^2}\right)\sqrt{\frac{\log m}{m}} \leq \frac{\phi^2}{\|\vct{x}\|_2^2} \leq 1 +\left(2+C_0\frac{\sigma}{\|\vct{x}\|_2^2}\right)\sqrt{\frac{\log m}{m}}+\frac{2 \log m}{m}
\]
for some numerical constant $C_0>0$. As a consequence, as long as $\frac{m}{\log m} \geq C(\delta)\left(1+ \frac{\sigma^2}{\|\vct{x}\|_2^4}\right)$, there holds
\[
\frac{9}{10}\leq 1-\delta \leq \frac{\phi^2}{\|\vct{x}\|_2^2} \leq 1+ \delta \leq \frac{11}{10}.
\]
\end{lemma}

\begin{proof}
By the definition of $\phi^2$ and $y_j, j=1, \ldots, m$, we have
\[
\phi^2 = \frac{1}{m} \sum_{j=1}^m (\vct{a}_j' \vct{x})^2 + \frac{1}{m} \sum_{j=1}^m\epsilon_j.
\]
As shown in Lemma \ref{lmm:noise}, with probability at least $1 - \frac{1}{m}$, 
\[
\left|\frac{1}{m} \sum_{j=1}^m\epsilon_j\right| \leq C_0 \sigma \sqrt{\frac{\log m}{m}}
\]
for some numerical constant $C_0>0$. Moreover, since $\vct{x}$ is fixed, there holds
\[
 \frac{\sum_{j=1}^m (\vct{a}_j' \vct{x})^2}{\|\vct{x}\|_2^2} \sim \chi^2(m).
 \]
 By Lemma 4.1 of \cite{Laurent00}, with probability at least $1 - \frac{2}{m}$, we have
 \[
 1 - 2\sqrt{\frac{\log m}{m}} \leq  \frac{\sum_{j=1}^m (\vct{a}_j' \vct{x})^2}{m\|\vct{x}\|_2^2} \leq 1 + 2\sqrt{\frac{\log m}{m}}+\frac{2 \log m}{m}.
 \]
 The proof is done.
\end{proof}

\begin{lemma}
\label{thm:initialization}
Let  $\alpha=K\left(1+\frac{\sigma}{\|\vct{x}\|_2^2}\right)$ for some large enough absolute constant $K$, and $\widehat{\vct{x}}^{(0)}$ be defined in \prettyref{algo:init}. There exists a random vector $\vct{x}^{(0)}$ satisfying $\vct{x}^{(0)} \indep \mtx{A}_{S^c}$ and $\supp{\vct{x}^{(0)}} \subset S$, such that on an event $E_{01}$ with probability at least $1 - \frac{16}{m} - 2e^{-k}$, we have
\[
\vct{x}^{(0)}=\widehat{\vct{x}}^{(0)},\text{~and~}{\min(\|\vct{x}^{(0)} - \vct{x}\|_2,  \|\vct{x}^{(0)} + \vct{x}\|_2)} \leq \frac{1}{6}{\|\vct{x}\|_2},
\]
provided $m\geq C \left(1+\frac{\sigma^2}{\|\vct{x}\|_2^4}\right)k^2 \log (mp)$. Here $C$ is an absolute constant.
\end{lemma}

\begin{proof}
Recall that $S=\{1, \ldots, k\}$ and $I_l=\frac{1}{m}\sum_{j=1}^m y_j a_{jl}^2$ for $l=1, \ldots, p$. Define
\begin{equation}
\label{eq:initial_support_oracle}
S_0 =\left\{ l \in S: I_l > \left(1+  \alpha \sqrt{\frac{\log (mp)}{m}}\right)\phi^2\right\} \subset S.
\end{equation}
Since $\{I_1, \ldots, l_k, \phi\} \indep \mtx{A}_{S^c}$, we have $S_0 \indep \mtx{A}_{S^c}$. Define $\vct{x}^{(0)}\in \mathbb{R}^p$ as the leading eigenvector of
\[
\mtx{W}_{S_0S_0}:=\frac{1}{m} \sum_{j=1}^m y_j\vct{a}_{j S_0} \vct{a}_{j S_0}' \in \mathbb{R}^{p \times p}
\]
with $2$-norm $\phi$. This easily implies $\supp{\vct{x}^{(0)}} \subset S_0 \subset S$. Since $\{\mtx{W}_{S_0S_0}, \phi\} \indep \mtx{A}_{S^c}$, we also have $\vct{x}^{(0)} \indep \mtx{A}_{S^c}$. 

To simplify notation, let us write for any $j\in [m]$, $\wt{y}_j := (\vct{a}_j' \vct{x})^2 = (\vct{a}_{jS}' \vct{x})^2$, which implies $y_j = \wt{y}_j + \eps_j$. Notice that
\begin{align}
	\label{eq:init-1}
I_l-\phi^2 = \frac{1}{m}\sum_{j=1}^m \wt{y}_j (a_{jl}^2 - 1)
+ \frac{1}{m}\sum_{j=1}^m \eps_j (a_{jl}^2 - 1),
\end{align}
in which we will first control the second term. For a given $l \in [p]$, we know  $a_{1l}^2 - 1, \ldots, a_{ml}^2-1$ are i.i.d. centered sub-exponential random variables with sub-exponential norms being an absolute constant. Then, by Bernstein inequality (see, e.g., Proposition 16 in \cite{vershyninNARMT}), we have with probability at least $1 - \frac{2}{mp}$,
\begin{align*}
\left|\sum_{j=1}^m \eps_j (a_{jl}^2 - 1)\right| \leq C_0\left(\|\vct{\eps} \|_2\sqrt{\log (mp)} + \|\vct{\eps} \|_\infty \log (mp) \right)
\end{align*}
for some absolute constant $C_0$. Then by Lemma \ref{lmm:noise}, with probability at least $1 - 4/m$, we have
\begin{equation}
\label{eq:init-noise}
\max_{1 \leq l \leq p}\left|\frac{1}{m}\sum_{j=1}^m \eps_j (a_{jl}^2 - 1)\right|\leq 
C_0\sigma \left(\sqrt{\frac{\log (mp)}{m}} + \frac{(\log m)(\log (mp))}{m}\right) \leq C_0\sigma \sqrt{\frac{\log (mp)}{m}}, 
\end{equation}
provided $m \geq C (\log p)$ for some absolute constant $C$.

%
%

Next, we prove that with high probability $\vct{x}^{(0)} = \widehat{\vct{x}}^{(0)}$. It suffices to prove $\widehat{S}_0=S_0$, i.e., $\widehat{S}_0 \subset S$.  For any $l\in S^c$, $a_{jl}$ and $\wt{y}_j$ are independent, and so conditional on $\{\wt{y}_j, j\in [m]\}$, $\sum_{j=1}^m \wt{y}_j a_{jl}^2$ is a weighted sum of $\chi^2_1$ variables.  By Lemma 4.1 of \cite{Laurent00}, 
\begin{align*}
\P \left\{ \sum_{j=1}^m \wt{y}_j(a_{jl}^2 - 1) > 2\sqrt{t}\left(\sum_{j=1}^m \wt{y}_j^2\right)^{\frac{1}{2}}  +  2 \left(\max_{j} \wt{y}_j\right)t  \right\} \leq \exp(-t).
\end{align*}
Moreover, Chebyshev's inequality, the Gaussian tail bound and the union bound lead to
\begin{align*}
\P\left\{\sum_{j=1}^m \wt{y}_j^2 / \|\vct{x}\|_2^4 > 3m + \sqrt{96m}t \right\} & \leq t^{-2}, \\
\P\left\{\max_j \wt{y}_j / \|\vct{x}\|_2^2 > t \right\} & \leq 2m\exp(-t/2).
\end{align*}
Thus, with probability at least $1 - \frac{4}{m}$, for all $l \in S^c$,
\begin{align}
\label{eq:init-2}
\frac{1}{m}\sum_{j=1}^m \wt{y}_j (a_{jl}^2 - 1)	\leq 2\sqrt{3+\sqrt{96}}\|\vct{x}\|_2^2 \sqrt{\log(mp)\over m} 
+ 8\|\vct{x}\|_2^2 \frac{(\log(mp))^2}{m} \leq 8 \|\vct{x}\|_2^2 \sqrt{\log(mp)\over m} .
\end{align}
Here the last inequality holds when $m\geq C$ for some absolute constant $C$.

Since $\alpha=K\left(1+\frac{\sigma}{\|\vct{x}\|_2^2}\right)$ with large enough $K$, by \eqref{eq:init-1}, \eqref{eq:init-2}, \eqref{eq:init-noise} and Lemma \ref{lmm:norm_estimation}, we obtain that with probability at least $1- \frac{11}{m}$, for all $l\in S^c$,
\begin{align*}
I_l-\phi^2 \leq (8\|\vct{x}\|_2^2 + C_0\sigma)\sqrt{\log(mp)\over m} \leq \alpha \phi^2\sqrt{\log(mp)\over m},
\end{align*}
which implies that $\wh{S}_0\subset S$.\\


Next, we prove that $\|\vct{x}^{(0)} - \vct{x}\|_2/\|\vct{x}\|_2 \leq \frac{1}{6}$ with high probability. For any fixed $l\in S$, straightforward calculation yields $\E \wt{y}_j a_{jl}^2 = \|\vct{x} \|_2^2 + 2x_l^2$. On the other hand, 
\begin{align*}
\E \wt{y}_j^2 a_{jl}^4 = 105 x_l^4 + 90 x_l^2 (\| \vct{x} \|_2^2 - x_l^2) + 9 (\| \vct{x} \|_2^2 -x_l^2)^2.
\end{align*}
So for $X_j = \| \vct{x} \|_2^2 + 2x_l^2 - \wt{y}_j a_{jl}^2$, we have $X_j\leq \| \vct{x} \|_2^2 + 2x_l^2 \leq 3\| \vct{x} \|_2^2$, $\E X_i = 0$ and $\E X_i^2 = 20x_j^4 + 68 \| \vct{x} \|_2^2 x_l^2 + 8\| \vct{x} \|_2^4 \leq 96\| \vct{x} \|_2^4$.
By Lemma \ref{lmm:one-sided-tail}, 
\begin{align*}
\P \left\{\sum_{j=1}^m \wt{y}_j a_{jl}^2 - m( \| \vct{x} \|_2^2 + 2x_l^2) \leq - t \right\} 
\leq \exp\left(-\frac{t^2}{192 \| \vct{x} \|_2^4 m} \right).
\end{align*}
Next, Lemma 4.1 of \cite{Laurent00} leads to with probability at least $1-\frac{1}{m}$,
\begin{align*}
\frac{1}{m}\sum_{j=1}^m \wt{y}_j - \| \vct{x} \|_2^2 \leq \pth{2\sqrt{\log m\over m} + \frac{2\log m}{m}}\|\vct{x}\|_2^2\leq 2.1\|\vct{x}\|_2^2\sqrt{\log m\over m}.
\end{align*}
The last two inequalities, together with \eqref{eq:init-noise} and \eqref{eq:init-1}, imply that with probability at least $1- \frac{6}{m}$, for all $l\in S$,
\begin{align*}
I_l-\phi^2 \geq 2x_l^2 - (16\|\vct{x} \|_2^2 + C_0\sigma) \sqrt{\log(mp)\over m}.
\end{align*}
Define $S_- = \left\{l\in S: x_l^2 \geq \left(11 + \frac{3}{5}\alpha\right) \|\vct{x}\|_2^2 \sqrt{\log (mp)\over m} \right\}$. Then, for all $l\in S_-$ we have
\begin{align*}
I_l-\phi^2 \geq (\frac{6}{5}\alpha \|\vct{x} \|_2^2  + 6\|\vct{x} \|_2^2 - C_0\sigma) \sqrt{\log(mp)\over m}.
\end{align*}
Since $\alpha=K\left(1+\frac{\sigma}{\|\vct{x}\|_2^2}\right)$ with sufficiently large absolute constant $K$, by lemma \ref{lmm:norm_estimation}, we have or all $l\in S_-$,
\begin{align*}
I_l-\phi^2 \geq \alpha \phi^2 \sqrt{\log(mp)\over m},
\end{align*}
with probability at least $1 - 9/m$. This implies $S_-  \subset S_0$.\\

Therefore, we have $\|\vct{x} - \vct{x}_{S_0}\|_2^2 \leq \|\vct{x} - \vct{x}_{S_{-}}\|_2^2 \leq (11+0.6\alpha)\|\vct{x}\|_2^2 \sqrt{\frac{k^2 \log (mp)}{m}} \leq \delta^2 \|\vct{x}\|_2^2$, provided that $m\geq C(\delta) \left(1+\frac{\sigma^2}{\|\vct{x}\|_2^4}\right)k^2 \log (mp)$. Notice that $\E \mtx{W} = \|\vct{x}\|_2^2 \mtx{I}_p + 2\vct{x}\vct{x}'$, which implies that $(\E \mtx{W})_{S S} = \|\vct{x}\|_2^2 (\mtx{I}_p)_{SS} + 2 \vct{x}\vct{x}'$. Furthermore, by the definition of $\mtx{W}$, we have
\[
\mtx{W}_{SS} = \frac{1}{m}\sum_{j=1}^m \left|{\vct{a}_j}_S'\vct{x}\right|^2 {\vct{a}_j}_S{\vct{a}_j}_S' + \frac{1}{m}\sum_{j=1}^m \epsilon_j {\vct{a}_j}_S{\vct{a}_j}_S'.
\]
By Lemma \ref{lmm:concentration}, with probability at least $1 - 1/m$, we have
\[
\left\|\frac{1}{m}\sum_{j=1}^m |{\vct{a}_j}_S' \vct{x}|^2 {\vct{a}_j}_S {\vct{a}_j}_S' - \left(\|\vct{x}\|_2^2 (\mtx{I}_p)_{SS}+2\vct{x}\vct{x}'\right)\right\| \leq \frac{\delta}{2} \|\vct{x}\|_2^2,
\]
 provided $m \geq C(\delta) k \log p$. Moreover, by Lemma \ref{lmm:noise} and Lemma \ref{lmm:concentration_noise}, with probability at least $1 - 2/m - 2e^{-k}$, we have $\norm{\sum_{j=1}^m \epsilon_j \vct{a}_{jS} \vct{a}_{jS}'} \leq C_0 \sigma \sqrt{m(k+\log m)}$. By assuming $m \geq C(\delta) \frac{\sigma^2}{\|\vct{x}\|_2^4} k \log (mp)$, we have $\frac{1}{m}\norm{\sum_{j=1}^m \epsilon_j \vct{a}_{jS} \vct{a}_{jS}'} \leq \frac{\delta}{2} \|\vct{x}\|_2^2$. This implies that
\[
\left\|\mtx{W}_{S_0 S_0} - (\E \mtx{W})_{S_0 S_0}\right\| \leq \left\|\mtx{W}_{S S} - (\E \mtx{W})_{S S}\right\| \leq \delta \|\vct{x}\|_2^2.
\]
It is noteworthy that the leading eigenvector of $(\E \mtx{W})_{SS}$ with unit norm is $\vct{x}_{S_0}/\|\vct{x}_{S_0}\|_2$, and the eigengap between the leading two eigenvalues of $(\E \mtx{W})_{S_0 S_0}$ is $2\|\vct{x}_{S_0}\|_2^2$. Recall that $\vct{x}^{(0)}$ is the leading eigenvector $\mtx{W}_{S_0 S_0}$ with norm $\phi$. Then by the Sin-Theta theorem,
\[
\left\|\frac{\vct{x}^{(0)}(\vct{x}^{(0)})^T}{\phi^2} - \frac{\vct{x}_{S_0}\vct{x}_{S_0}^T}{\|\vct{x}_{S_0}\|_2^2}\right\|\leq \frac{\delta \|\vct{x}\|_2^2}{2\|\vct{x}_{S_0}\|_2^2 - \delta\|\vct{x}\|_2^2} \leq \frac{\delta}{2-5\delta}.
\]
By Lemma \ref{lmm:norm_estimation}, we have $1+\delta \geq \phi/\|\vct{x}\|_2\geq 1- \delta$. Together with $1 \geq \|\vct{x}_{S_0}\|_2/\|\vct{x}\|_2 \geq 1- \delta$, we can easily obtain that $\min(\|\vct{x}^{(0)} - \vct{x}\|_2, \|\vct{x}^{(0)} + \vct{x}\|_2) \leq C_0 \delta \|\vct{x}\|_2$ for some absolute constant $C_0$. By letting $\delta$ be small enough, we have $\min(\|\vct{x}^{(0)} - \vct{x}\|_2, \|\vct{x}^{(0)} + \vct{x}\|_2) \leq 1/6 \|\vct{x}\|_2$.

In conclusion, we have
\[
\P\left(\vct{x}^{(0)}=\widehat{x}^{(0)} \text{~and~} {\min(\|\vct{x}^{(0)} - \vct{x}\|_2, \|\vct{x}^{(0)} + \vct{x}\|_2)}\leq 1/6{\|\vct{x}\|_2}  \right) \geq 1- \frac{16}{m} - 2e^{-k}.
\]
\end{proof}

\begin{lemma}
\label{thm:contraction}
Define $\eta(\vct{z}) = \mathcal{T}_{\frac{\mu}{\phi^2}\tau(\vct{z})} \left(\vct{z} - \frac{\mu}{\phi^2} \nabla f (\vct{z})_S\right)$. With probability at least $1 - \frac{15}{m} - 4e^{-k}$, for all $\vct{z} \in \mathbb{R}^p$ satisfying $\|\vct{z} - \vct{x}\|_2 \leq \frac{1}{6}\|\vct{x}\|_2$ and $\supp{\vct{z}} \subset S$, we have
\[
\frac{\|\eta(\vct{z}) - \vct{x}\|_2}{\|\vct{x}\|_2} \leq \left(1- \frac{\mu}{8}\right)\frac{\|\vct{z} - \vct{x}\|_2}{\|\vct{x}\|_2} + C_0\frac{\mu \sigma}{\|\vct{x}\|_2^2}  \sqrt{\frac{k\log p}{m}},
\]
provided $\mu \leq \mu_0$ and $m\geq C k^2 \log p$. Here $C_0$, $C$, and $\mu_0$ are numerical constants. This implies that, on an event $E_{02}$ with probability at least $1 - \frac{30}{m} - 8e^{-k}$, for all $\vct{z} \in \mathbb{R}^p$ satisfying $\min(\|\vct{z} - \vct{x}\|_2, \|\vct{z} +\vct{x}\|_2) \leq \frac{1}{6}\|\vct{x}\|_2$ and $\supp{\vct{z}} \subset S$, we have
\[
{\min(\|\eta(\vct{z}) - \vct{x}\|_2, \|\eta(\vct{z}) + \vct{x}\|_2)} \leq \left(1- \frac{\mu}{8}\right) {\min(\|\vct{z} - \vct{x}\|_2, \|\vct{z} +\vct{x}\|_2)} + C_0\frac{\mu \sigma}{\|\vct{x}\|_2}  \sqrt{\frac{k\log p}{m}}.
\]

\end{lemma}

\begin{proof}
For $\vct{z}$ supported on $S$, define
\[
\vct{u} = \eta(\vct{z}) = \mathcal{T}_{\frac{\mu}{\phi^2}\tau(\vct{z})} \left(\vct{z} - \frac{\mu}{\phi^2} \nabla f (\vct{z})_S\right) = \vct{z} - \frac{\mu}{\phi^2} \nabla f(\vct{z})_S + \frac{\mu}{\phi^2} \tau(\vct{z}) \vct{v},
\]
where $\vct{v} \in \mathbb{R}^p$, $\supp{\vct{v}} \subset S$ and $\|\vct{v}\|_\infty \leq 1$. \\

Since $\supp{\vct{z}}\subset S=\{1, \ldots, k\}$, we have
\begin{align}
\nabla f(\vct{z})_S = \frac{1}{m} \sum_{j=1}^m \left(|{\vct{a}_j}_S'\vct{z}|^2 - y_j\right)({\vct{a}_j}_S' \vct{z}){\vct{a}_j}_S.	
\end{align}
For convenience, let 
\begin{align}
\wt{\nabla f(\vct{z})}_S = \frac{1}{m} \sum_{j=1}^m \left(|{\vct{a}_j}_S'\vct{z}|^2 - |{\vct{a}_j}_S'\vct{x}|^2 \right)({\vct{a}_j}_S' \vct{z}){\vct{a}_j}_S,
\end{align}
and so
\begin{align}
\label{eq:T2-exp}
\nabla f(\vct{z})_S - \wt{\nabla f(\vct{z}) }_S = - \frac{1}{m}\sum_{j=1}^m \eps_j ({\vct{a}_j}_S' \vct{z}){\vct{a}_j}_S.
\end{align}
Denote $\vct{h} = \vct{z} - \vct{x} \in \mathbb{R}^p$, which implies $\supp{\vct{h}} \subset S$ and $\|\vct{h}\|_2\leq \|\vct{x}\|_2/6$. Straightforward calculation yields 
\begin{align}
\label{eq:onestep-1}
\|\vct{u} - \vct{x}\|_2 &\leq \left\|\vct{h} - \frac{\mu}{\phi^2} \wt{\nabla f(\vct{z})}_S\right\|_2+\frac{\mu}{\phi^2}\left\|\nabla f(\vct{z})_S - \wt{\nabla f(\vct{z}) }_S\right\|_2 + \frac{\mu \sqrt{k}}{\phi^2} \tau(\vct{z}) \nonumber
\\
 &:= T_1 + \frac{\mu}{\phi^2}T_2 + \frac{\mu \sqrt{k}}{\phi^2} \tau(\vct{z}).
\end{align}
It suffices to bound $T_1$, $T_2$ and $\tau(\vct{z})$.


\paragraph{Bound for $T_1$}
By simple algebra, we have
\begin{align}
\label{eq:T1-decompose}
T_1^2 &=\norm{\vct{h}}_2^2 - \frac{\mu}{\phi^2}\frac{1}{m} \sum_{j=1}^m \left(2({\vct{a}_j}_S'\vct{x})^2({\vct{a}_j}_S'\vct{h})^2 + 3 ({\vct{a}_j}_S'\vct{x})({\vct{a}_j}_S'\vct{h})^3 + ({\vct{a}_j}_S'\vct{h})^4\right)+ \frac{\mu^2}{\phi^4} \norm{\wt{\nabla f(\vct{z}) }_S}_2^2 \nonumber
\\	
&:= \norm{\vct{h}}_2^2 - \frac{\mu}{\phi^2} T_{11}  + \frac{\mu^2}{\phi^4} T_{12}.
\end{align}
In what follows, we derive lower bound for $T_{11}$ and upper bound for $T_{12}$ separately.\\

Notice that 
\[
T_{11}=\frac{1}{m} \sum_{j=1}^m \left(2 ({\vct{a}_j}_S'\vct{x})^2({\vct{a}_j}_S'\vct{h})^2 + 3 ({\vct{a}_j}_S'\vct{x})({\vct{a}_j}_S'\vct{h})^3 + ({\vct{a}_j}_S'\vct{h})^4\right).
\]
First, by Lemma \ref{lmm:concentration} with probability at least $1-1/m$, we have
\[
\frac{1}{m} \sum_{j=1}^m 2 ({\vct{a}_j}_S'\vct{x})^2({\vct{a}_j}_S'\vct{h})^2 \geq (2 - 2\delta)\left(2(\vct{x}'\vct{h})^2+\|\vct{x}\|_2^2\|\vct{h}\|_2^2\right).
\]
By Lemma \ref{lmm:universal_upperbound}, with probability at least $1 - 2/m$, we have
\begin{align*}
\frac{1}{m} \sum_{j=1}^m 3 ({\vct{a}_j}_S'\vct{x})({\vct{a}_j}_S'\vct{h})^3 &\leq  \frac{3}{m} \left(\sum_{j=1}^m ({\vct{a}_j}_S'\vct{x})^4\right)^{\frac{1}{4}}\left(\sum_{j=1}^m ({\vct{a}_j}_S'\vct{h})^4\right)^{\frac{3}{4}}
\\
& \leq \frac{3}{m}((3m)^{\frac{1}{4}}+k^{\frac{1}{2}}+\sqrt{2\log m})^4\|\vct{x}\|_2\|\vct{h}\|_2^3
\\
& \leq 10\|\vct{x}\|_2\|\vct{h}\|_2^3,
\end{align*}
provided $m \geq Ck^2$ for some sufficiently large numerical constant $C$. This implies
\[
T_{11} \geq (2- 2\delta)\|\vct{x}\|_2^2\|\vct{h}\|_2^2 - 10\|\vct{x}\|_2\|\vct{h}\|_2^3 \geq (1/3- 2\delta)\|\vct{x}\|_2^2\|\vct{h}\|_2^2.
\]
As to the upper bound for $T_{12}$, we can find $\|\vct{w}\|_2=1$, such that
\begin{align*}
T_{12}=\|\wt{\nabla f (\vct{z})}_S\|_2^2 &\leq \frac{2}{m^2}\left|\sum_{j=1}^m |{\vct{a}_j}_S' \vct{h}||{\vct{a}_j}_S' (2\vct{x}+\vct{h})||{\vct{a}_j}_S' (\vct{x}+\vct{h})||{\vct{a}_j}_S' \vct{w}|\right|^2.
\end{align*}
By Holder's inequality and Lemma \ref{lmm:universal_upperbound}, we have
\begin{align*}
T_{12} &\leq \frac{2}{m^2}\left(\sum_{j=1}^m |{\vct{a}_j}_S' \vct{h}|^4\right)^{\frac{1}{2}} \left(\sum_{j=1}^m |{\vct{a}_j}_S' (2\vct{x}+\vct{h})|^4\right)^{\frac{1}{2}} \left(\sum_{j=1}^m |{\vct{a}_j}_S' (\vct{x}+\vct{h})|^4\right)^{\frac{1}{2}} \left(\sum_{j=1}^m |{\vct{a}_j}_S' \vct{w}|^4\right)^{\frac{1}{2}}
\\
&\leq \frac{2}{m^2} ((3m)^{\frac{1}{4}}+k^{\frac{1}{2}}+\sqrt{2\log m})^8 \|\vct{h}\|_2^2\|2\vct{x}+\vct{h}\|_2^2\|\vct{x}+\vct{h}\|_2^2\|\vct{w}\|_2^2 \leq C_0 \|\vct{h}\|_2^2\|\vct{x}\|_2^4,
\end{align*}
provided $m \geq C k^2$, with sufficiently large constants $C_0$ and $C$. To summarize, with probability at least $1- 3/m$, 
\begin{align}
	\label{eq:T1-bd}
T_1^2 & \leq \norm{\vct{h}}_2^2 - \frac{\mu}{\phi^2}(1/3 - 2\delta)\norm{\vct{h}}_2^2\norm{\vct{x}}_2^2 + C_0\frac{\mu^2}{\phi^4}\norm{\vct{x}}_2^4 \norm{\vct{h}}_2^2.
\end{align}
By Lemma \ref{lmm:norm_estimation}, letting $\delta$ small enough, we have with probability at least $1 - 6/m$,
\[
T_1 \leq (1 - \mu/8)\norm{\vct{h}}_2,
\]
provided $\mu \leq \mu_0$ with sufficiently small absolute constant $\mu_0>0$.

\paragraph{Bound for $T_2$}
Note that 
\begin{align*}
T_{2} \leq \frac{7}{6m}\norm{\vct{x}}_2 \norm{\sum_{j=1}^m \eps_j \vct{a}_{jS}\vct{a}_{jS}'}.
\end{align*}
By Lemma \ref{lmm:noise} and Lemma \ref{lmm:concentration_noise}, with probability at least $1 - 2/m - 2e^{-k}$,  we have
\begin{align*}
\norm{\sum_{j=1}^m \epsilon_j \vct{a}_{jS} \vct{a}_{jS}'} \leq C_0 \sigma \sqrt{m(k+\log m)}
\end{align*}
provided $m/\log m \geq k$. In summary, by Lemma \ref{lmm:norm_estimation}, we have that with probability at least $1 - 5/m - 2e^{-k}$,
\begin{align*}
\frac{\mu}{\phi^2}T_{2} \leq C_0 \mu \frac{\sigma}{\norm{\vct{x}}_2}\sqrt{\frac{k + \log m}{m}}.
\end{align*}


\paragraph{Bound for $\tau(\vct{z})$}
By simple algebra, 
\begin{align*}
\tau^2(\vct{z}) &= \frac{\beta \log p}{m^2} \sum_{j=1}^m \left(({\vct{a}_j}_S' \vct{h}){\vct{a}_j}_S' (2\vct{x}+\vct{h}) - \epsilon_j\right)^2|{\vct{a}_j}_S' (\vct{x}+\vct{h})|^2
\\
&\leq  \frac{2\beta \log p}{m^2} \left\{\sum_{j=1}^m |{\vct{a}_j}_S' \vct{h}|^2|{\vct{a}_j}_S' (2\vct{x}+\vct{h})|^2|{\vct{a}_j}_S' (\vct{x}+\vct{h})|^2 + \sum_{j=1}^m \epsilon_j^2|{\vct{a}_j}_S' (\vct{x}+\vct{h})|^2\right\}
\\
&:=\frac{2\beta \log p}{m^2}(\mathcal{T}_1 + \mathcal{T}_2).
\end{align*}
By Holder's inequality and Lemma \ref{lmm:universal_upperbound}, with probability at least $1 - 2/m$, we have
\begin{align*}
\mathcal{T}_1 &\leq \left(\sum_{j=1}^m |{\vct{a}_j}_S' \vct{h}|^6\right)^{\frac{1}{3}} \left(\sum_{j=1}^m |{\vct{a}_j}_S' (2\vct{x}+\vct{h})|^6\right)^{\frac{1}{3}} \left(\sum_{j=1}^m |{\vct{a}_j}_S' (\vct{x}+\vct{h})|^6\right)^{\frac{1}{3}}
\\
&\leq C_0 \|\mtx{A}_S\|_{2 \rightarrow 6}^6 \|\vct{h}\|_2^2 \|\vct{x} \|_2^4 \leq C_0(m + k^3)\|\vct{h}\|_2^2\|\vct{x}\|_2^4
\end{align*}
for some numerical constant $C_0$. By Lemma \ref{lmm:noise} and Lemma \ref{lmm:concentration_noise}, with probability at least $1 - 2/m - 2e^{-k}$, we have, 
\[
\mathcal{T}_2 \leq\frac{49}{36} \|\vct{x}\|_2^2 \norm{\sum_{j=1}^m \epsilon_j^2 \vct{a}_{jS} \vct{a}_{jS}'}  \leq C_0 m \sigma^2 \|\vct{x}\|_2^2,
\]
for some numerical constant $C_0$, provided $\frac{m}{\log^2 m}\geq k$. In summary, 
\begin{align}
\label{eq:tau-bd}
\frac{\mu}{\phi^2}\sqrt{k}\tau \leq C_0 \mu \left(\frac{\sqrt{(mk+k^4)\log p}}{m}\|\vct{h}\|_2 + \frac{\sigma}{\|\vct{x}\|_2}\sqrt{\frac{k \log p}{m}}\right) \leq \frac{\mu \|\vct{h}\|_2}{16} + C_0 \frac{\mu \sigma}{\|\vct{x}\|_2}\sqrt{\frac{k \log p}{m}},
\end{align}
provided $m \geq C \max(k \log p, k^2\sqrt{\log p})$.

\paragraph{Summary}
We can guarantee that, with probability at least $1 - \frac{15}{m} - 4e^{-k}$,
\begin{align}
\frac{\norm{\vct{u} - \vct{x}}_2}{\|\vct{x}\|_2} \leq \left(1-\frac{\mu}{16}\right)\frac{\norm{\vct{z} - \vct{x}}_2}{\|\vct{x}\|_2} + C_0 \mu \sqrt{\frac{k \log p}{m}}\frac{\sigma}{\|\vct{x}\|_2^2},
\end{align}
for some absolute constant $C_0 > 0$, provided $m \geq C k^2 \log (mp)$ and $\mu \leq \mu_0$.
\end{proof}

Suppose $E_0$ is the intersection of the events $E_{01}$ and $E_{02}$ described by Lemmas \ref{thm:initialization} and \ref{thm:contraction}, respectively. Then we have 
\[
\P(E_0)\geq 1 - \frac{46}{m} - 10e^{-k}.
\] 
The following induction argument guarantees the effectiveness of thresholded Wirtinger flow:
\begin{lemma}
\label{lmm:induction}
Let $\beta=4$ and $\widehat{\vct{x}}^{(n)}, n=0,1,2,\ldots$ are defined iteratively by \eqref{eq:x-init} and \eqref{eq:iteration}. For fixed $n\geq 0$, assume that there exists a random vector $\vct{x}^{(n)}$ satisfying $\vct{x}^{(n)} \indep \mtx{A}_{S^c}$ and $\supp{\vct{x}^{(n)}} \subset S$, and that on an event $E_n \subset E_0$ we have $\widehat{\vct{x}}^{(n)} = \vct{x}^{(n)}$ and $\min\limits_{i=0,1}\|\widehat{\vct{x}}^{(n)} - (-1)^i\vct{x}\|_2 \leq \frac{1}{6} {\|\vct{x}\|_2}$. Then there exists a random vector $\vct{x}^{(n+1)}$ satisfying $\vct{x}^{(n+1)} \indep \mtx{A}_{S^c}$ and $\supp{\vct{x}^{(n+1)}} \subset S$, and on an event $E_{n+1} \subset E_n$ satisfying $\P(E_n/E_{n+1})\leq 1 - \frac{1}{m^2p}$, we have $\widehat{\vct{x}}^{(n+1)} = \vct{x}^{(n+1)}$ and 
\[
{\min\limits_{i=0,1}\|\widehat{\vct{x}}^{(n+1)} - (-1)^i\vct{x}\|_2}\leq \left(1-\frac{\mu}{16}\right) {\min\limits_{i=0,1}\|\widehat{\vct{x}}^{(n)} - (-1)^i\vct{x}\|_2} + C_0\frac{\mu \sigma}{\|\vct{x}\|_2}  \sqrt{\frac{k\log p}{m}} \leq \frac{1}{6}\|\vct{x}\|_2,
\]
provided $m\geq C \left(1+\frac{\sigma^2}{\|\vct{x}\|_2^4}\right)k^2 \log (mp)$ for sufficiently large $C$.
\end{lemma}

\begin{proof}
The improved estimation is defined as
\[
\widehat{\vct{x}}^{(n+1)} = \mathcal{T}_{\frac{\mu}{\phi^2}\tau(\widehat{\vct{x}}^{(n)})} \left(\widehat{\vct{x}}^{(n)} - \frac{\mu}{\phi^2} \nabla f (\widehat{\vct{x}}^{(n)})\right).
\]
where $\mathcal{T}_\tau$ is the soft-thresholding operator. We now define
\[
\vct{x}^{(n+1)} :=\eta(\vct{x}^{(n)}) = \mathcal{T}_{\frac{\mu}{\phi^2}\tau(\vct{x}^{(n)})} \left(\vct{x}^{(n)} - \frac{\mu}{\phi^2} \nabla f(\vct{x}^{(n)})_S\right).
\]
By the definition of $\nabla f$, $\tau$ and $\phi$, as well as the assumption that $\vct{x}^{(n)} \indep \mtx{A}_{S^c} \text{~and~}\supp{\vct{x}^{(n)}} \subset S$, we can prove $\supp{\vct{x}^{(n+1)}} \subset S$ as well as $\vct{x}^{(n+1)} \indep \mtx{A}_{S^c}$. In fact, by the definition \eqref{eq:threshold}, we know if $\vct{x}^{(n)}$ is supported on $S$ and independent of $\mtx{A}_{S^c}$, then $\tau(\vct{x}^{(n)})$ is independent of $\mtx{A}_{S^c}$. Moreover, by the definition of the gradient \eqref{eq:gradient}, we know $ \left(\nabla f (\vct{x}^{(n)})\right)_S$ is supported on $S$ and independent of $\mtx{A}_{S^c}$. The assertion is established by the obvious fact $\phi \indep \mtx{A}_{S^c}$ shown in Lemma \ref{lmm:independence}.

In the following, we will construct $E_{n+1} \subset E_n$ such that $\widehat{\vct{x}}^{(n+1)}=\vct{x}^{(n+1)}$ on $E_{n+1}$. For any $i=k+1, k+2, \ldots, p$, with probability $1 - \frac{1}{m^2p^2}$,
\begin{align*}
\left|\frac{\partial}{\partial z_i} f (\vct{x}^{(n)})\right|&=\left|\frac{1}{m} \sum_{j=1}^m \left(|{\vct{a}_j}'\vct{x}^{(n)}|^2 - y_j\right)({\vct{a}_j}' \vct{x}^{(n)})(\vct{a}_j)_i\right|
\\
& \leq \frac{\sqrt{4 \log (mp)}}{m}\sqrt{\sum_{j=1}^m \left(|{\vct{a}_j}'\vct{x}^{(n)}|^2 - y_j\right)^2|{\vct{a}_j}' \vct{x}^{(n)}|^2}
\\
& \leq \tau(\vct{x}^{(n)}).
\end{align*}
The first inequality is due to $\supp{\vct{x}^{(n)}} \subset S$ and $\vct{x}^{(n)} \indep \mtx{A}_{S^c}$, and the second inequality is due to $\beta = 4$. Then with probability at least $1 - \frac{1}{m^2 p}$,
\[
\max_{k+1 \leq i \leq p}\left|\frac{\partial}{\partial z_i} f (\vct{x}^{(n)})\right| \leq \tau(\vct{x}^{(n)}),
\]
which implies
\[
\mathcal{T}_{\frac{\mu}{\phi^2}\tau(\vct{x}^{(n)})} \left(\vct{x}^{(n)} - \frac{\mu}{\phi^2} \nabla f(\vct{x}^{(n)})\right)= \mathcal{T}_{\frac{\mu}{\phi^2}\tau(\vct{x}^{(n)})} \left(\vct{x}^{(n)} - \frac{\mu}{\phi^2} \nabla f(\vct{x}^{(n)})_S\right).
\]
Notice that on the event $E_n$, we have $\widehat{\vct{x}}^{(n)} = \vct{x}^{(n)}$, and hence
\[
\widehat{\vct{x}}^{(n+1)} = \mathcal{T}_{\frac{\mu}{\phi^2}\tau(\vct{x}^{(n)})} \left(\vct{x}^{(n)} - \frac{\mu}{\phi^2} \nabla f (\vct{x}^{(n)})\right).
\]
Then there exists $E_{n+1} \subset E_n$, such that $\P(E_n/E_{n+1}) \leq \frac{1}{m^2 p}$, and 
\[
\widehat{\vct{x}}^{(n+1)} = \mathcal{T}_{\frac{\mu}{\phi^2}\tau(\vct{x}^{(n)})} \left(\vct{x}^{(n)} - \frac{\mu}{\phi^2} \nabla f(\vct{x}^{(n)})_S\right) = \vct{x}^{(n+1)}.
\]
By the assumption, we have
\[
{\min(\|\vct{x}^{(n)} - \vct{x}\|_2, \|\vct{x}^{(n)} + \vct{x}\|_2)} \leq \frac{1}{6}{\|\vct{x}\|_2} \text{~on~} E_n.
\]
Since $E_n \subset E_0$ and $\vct{x}^{(n+1)} =\eta(\vct{x}^{(n)})$, by Lemma \ref{thm:contraction}, we have
\begin{align*}
&{\min(\|\vct{x}^{(n+1)} - \vct{x}\|_2, \|\vct{x}^{(n+1)} + \vct{x}\|_2)} 
\\
&\leq \left(1-\frac{\mu}{16}\right){\min(\|\vct{x}^{(n)} - \vct{x}\|_2, \|\vct{x}^{(n)} + \vct{x}\|_2)} + C_0\frac{\mu \sigma}{\|\vct{x}\|_2}  \sqrt{\frac{k\log p}{m}}  \leq \frac{1}{6}\|\vct{x}\|_2  \text{~on~}E_n,
\end{align*}
provided $m \geq C (\sigma^2/\|\vct{x}\|_2^4) k\log p$ for a sufficiently large absolute constant $C$. Since $E_{n+1} \subset E_n$, and $\widehat{\vct{x}}^{(n+1)} = \vct{x}^{(n+1)}$ on $E_{n+1}$, we have
\begin{align*}
{\min\limits_{i=0,1}\|\widehat{\vct{x}}^{(n+1)} - (-1)^i\vct{x}\|_2}\leq \left(1-\frac{\mu}{16}\right) {\min\limits_{i=0,1}\|\widehat{\vct{x}}^{(n)} - (-1)^i\vct{x}\|_2} + C_0\frac{\mu \sigma}{\|\vct{x}\|_2}  \sqrt{\frac{k\log p}{m}} \leq \frac{1}{6} \|\vct{x}\|_2\text{~on~}E_{n+1}.
\end{align*}
\end{proof}

Theorem \ref{thm:upper} can be directly implied by Lemma \ref{lmm:induction}. In fact, by Lemma \ref{thm:initialization}, we know the initial condition in \ref{lmm:induction} holds. For all $t=1, 2, 3, \ldots$, straight forward calculation yields
\[
\frac{\min(\|\widehat{\vct{x}}^{(t)} - \vct{x}\|_2, \|\widehat{\vct{x}}^{(t)} + \vct{x}\|_2)}{\|\vct{x}\|_2} \leq \frac{1}{6}\left(1 - \frac{\mu}{16}\right)^t + C_0\frac{\sigma}{\|\vct{x}\|_2^2}\sqrt{\frac{k \log p}{m}} \text{~on~} E_t
\]
for some universal constant $C_0$, where $\P(E_t) \geq 1 - \frac{46}{m} - 10e^{-k} - \frac{t}{mp^2}$. 

\appendix
\section{Preliminaries and supporting lemmas}

\begin{lemma}(\cite{Bentkus03})
\label{lmm:one-sided-tail}
Suppose $X_1,\dots, X_m$ are i.i.d.~real-valued random variables obeying $X_i\leq b$ for some absolute constant $b>0$, $\E X_i = 0$ and $\E X_i^2 = v^2$. Setting $\sigma^2 = m (b^2 \vee v^2)$,
\begin{align*}
\P \sth{X_1+\cdots + X_m \geq y} \leq \exp\pth{- {y^2\over 2\sigma^2}} \wedge c_0(1 - \Phi(y/\sigma))
\end{align*}
where one can take $c_0 = 25$.
\end{lemma}

\begin{lemma} (Proposition 34 \cite{vershyninNARMT})
\label{lmm:gaussian_concentration}
Suppose that $\vct{x} \sim \mathcal{N}(0, \mtx{I}_n)$ is a standard normal random vector, and $f: \mathbb{R}^n \rightarrow \mathbb{R}$ is a $1$-Lipschitz function. Then
\[
\P(f(\vct{x}) - \E f(\vct{x}) \geq t)\leq e^{-\frac{t^2}{2}}.
\]
\end{lemma}

\begin{lemma} (Proposition 33 \cite{vershyninNARMT})
\label{lmm:slepian}
Consider two centered Gaussian processes $(X_t)_{t \in T}$ and $(Y_t)_{t \in T}$ whose increments satisfy the inequality
\[
\E |X_s - X_t|^2 \leq \E |Y_s - Y_t|^2
\]
for all $s, t \in T$. Then
\[
\E \sup_{t \in T} X_t \leq \E \sup_{t \in T} Y_t.
\]
\end{lemma}

\begin{lemma} (Proposition 35 \cite{vershyninNARMT})
Let $\mtx{A}_S \in \mathbb{R}^{m \times p}$ be defined in \eqref{eq:A_S}. Then, with probability at least $1- 2 \exp(-t^2/2)$, we have the following inequality
\begin{equation}
\label{eq:second_moment}
\|\mtx{A}_S\| \leq \sqrt{m} + \sqrt{k} + t.
\end{equation}
\end{lemma}

\begin{lemma}
\label{lmm:universal_upperbound}
Let $\mtx{A}_S \in \mathbb{R}^{m \times p}$ be defined in \eqref{eq:A_S}. Then, with probability at least $1- 4 \exp(-t^2/2)$, the following inequalities hold
\begin{equation}
\label{eq:sixth_moment}
\|\mtx{A}_S\|_{2 \rightarrow 6} \leq (15 m)^{1/6} + \sqrt{k} +t,
\end{equation}
and
\begin{equation}
\label{eq:fourth_moment}
\|\mtx{A}_S\|_{2 \rightarrow 4} \leq (3 m)^{1/4} + \sqrt{k} +t.
\end{equation}
\end{lemma}

\begin{proof}
The proof follows that of Theorem 32 in \cite{vershyninNARMT} step by step. Define $X_{\vct{u}, \vct{v}}=\langle \mtx{A}_S \vct{u}, \vct{v} \rangle$ on 
\[
T = \{(\vct{u}, \vct{v}): \vct{u} \in \mathbb{R}^p, \supp{U} \subset S, \|\vct{u}\|_2=1, \vct{v} \in \mathbb{R}^m, \|\vct{v}\|_{6/5}=1\}.
\]
Then $\|\mtx{A}_S\|_{2 \rightarrow 6} = \max_{(\vct{u}, \vct{v}) \in T} X_{\vct{u}, \vct{v}}$. Define 
\[
Y_{\vct{u}, \vct{v}} = \langle \vct{g}_S, \vct{u}\rangle + \langle \vct{h}, \vct{v} \rangle
\]
where $\vct{g}_S \in \mathbb{R}^p$ with $\supp{\vct{g}_S} = S$ and $\vct{h} \in \mathbb{R}^m$ are independent standard Gaussian random vectors.\\

For any $(\vct{u}, \vct{v}), (\vct{u}', \vct{v}') \in T$, we have
\[
\E|X_{\vct{u}, \vct{v}} - X_{\vct{u}', \vct{v}'}|=\|\vct{v}\|_2^2 + \|\vct{v}'\|_2^2 - 2\langle \vct{u}, \vct{u}'\rangle \langle \vct{v}, \vct{v}'\rangle 
\]
and
\[
\E|Y_{\vct{u}, \vct{v}} - Y_{\vct{u}', \vct{v}'}|=2+\|\vct{v}\|_2^2 + \|\vct{v}'\|_2^2 - 2\langle \vct{u}, \vct{u}' \rangle - \langle \vct{v}, \vct{v}'\rangle. 
\]
Therefore, 
\[
\E|X_{\vct{u}, \vct{v}} - X_{\vct{u}', \vct{v}'}| - \E|Y_{\vct{u}, \vct{v}} - Y_{\vct{u}', \vct{v}'}| = 2(1-\langle \vct{u}, \vct{u}' \rangle  )(1 - \langle \vct{v}, \vct{v}'\rangle) \geq 0,
\]
due to $\|\vct{u}\|_2=\|\vct{u}'\|_2=1$, $\|\vct{v}\|_2 \leq \|\vct{v}\|_{6/5}=1$, and $\|\vct{v}'\|_2 \leq \|\vct{v}'\|_{6/5}=1$. Then by Lemma \ref{lmm:slepian}, we have
\[
\E \|\mtx{A}_S\|_{2 \rightarrow 6} \leq \E \max_{(\vct{u}, \vct{v}) \in T} Y_{\vct{u}, \vct{v}} = \E\|\vct{g}_S\|_2 + \E \|\vct{h}\|_6 \leq \sqrt{\E \|\vct{g}_S\|_2^2} + (\E \|\vct{h}\|_6^6)^{1/6} = \sqrt{k} + (15m)^{1/6}.
\]
Since $\|\cdot\|_{2 \rightarrow 6}$ is a $1$-Lipschitz function, by Lemma \ref{lmm:gaussian_concentration}, there holds with probability at least $1 - 2 \exp(-t^2/2)$
\[
\|\mtx{A}_S\|_{2 \rightarrow 6} \leq \sqrt{k} + (15m)^{1/6} +t.
\]
Similarly, with probability at least $1 - 2 \exp(-t^2/2)$
\[
\|\mtx{A}_S\|_{2 \rightarrow 4} \leq \sqrt{k} + (3m)^{1/4} +t.
\]
\end{proof}

\begin{lemma} 
\label{lmm:concentration}
On an event with probability at least $1 - 1/m$, we have
\[
\left\|\frac{1}{m}\sum_{j=1}^m |{\vct{a}_j}_S' \vct{x}|^2 {\vct{a}_j}_S {\vct{a}_j}_S' - \left(\|\vct{x}\|_2^2 (\mtx{I}_p)_S+2\vct{x}\vct{x}'\right)\right\| \leq \delta \|\vct{x}\|_2^2
\]
provided $m\geq C(\delta)k \log k$, where $C(\delta)$ is constant only depending on $\delta$. Here $(\mtx{I}_p)_S$ by definition is a diagonal matrix with first $k$ diagonal entries equal to $1$, whereas other entries being $0$. Furthermore, it implies that 
\[
\frac{1}{m} \sum_{j=1}^m ({\vct{a}_j}_S' \vct{x})^2({\vct{a}_j}_S' \vct{h})^2 \geq 2(\vct{x}'\vct{h})^2 + (1- \delta)\|\vct{x}\|_2^2\|\vct{h}\|_2^2
\]
for any $\vct{h} \in \mathbb{R}^p$ that satisfies $\supp{\vct{h}} \subset S$. 
\end{lemma}

The proof of this lemma is the same as that of Lemma 7.4 in \cite{Candes14}.

\begin{lemma}
\label{lmm:noise}
Suppose $\epsilon_1, \ldots, \epsilon_m$ are independent zero-mean sub-exponential random variables with
\[
\sigma := \max_{1 \leq i \leq m} \| \epsilon_i \|_{\psi_1}.
\]
Then with probability at least $1 - \frac{3}{m}$, we have 
\[
\left|\frac{1}{m} \sum_{j=1}^m\epsilon_j\right| \leq C_0 \sigma \sqrt{\frac{\log m}{m}}, \quad \|\vct{\epsilon}\|_\infty \leq C_0  \sigma \log m, \quad \left|\frac{1}{m} \sum_{j=1}^m\epsilon_j^2\right| \leq C_0 \sigma^2, \quad \text{and~} \left|\frac{1}{m} \sum_{j=1}^m\epsilon_j^4\right| \leq C_0 \sigma^4.
\]
provided $m \geq m_0$ for some numerical constants $C_0$ and $m_0$.
\end{lemma}

\begin{proof}
By Proposition 16 in \cite{vershyninNARMT}, we have
\[
\P\left(\left|\sum_{i=1}^m \epsilon_i\right|\geq t\right) \leq 2 \exp\left[-c \min \left(\frac{t^2}{m \sigma^2}, \frac{t}{\sigma}\right)\right].
\]
This implies that with probability at least $1 - \frac{2}{m^{10}}$, we have
\[
\left|\sum_{i=1}^m \epsilon_i\right| \leq C_0\sigma\max\left(\sqrt{m \log m}, \log m\right)\leq C_0\sigma \sqrt{m \log m}
\]
provided $m \geq m_0$. This implies that
\[
\left|\frac{1}{m} \sum_{j=1}^m\epsilon_j\right| \leq C_0 \sigma \sqrt{\frac{\log m}{m}}.
\]
~\\
By the basic properties of sub-exponential random variables, for each $j=1, \ldots, m$, we have
\[
\P\left(|\epsilon_j|\geq t\right) \leq \exp \left(1 - c \frac{t}{\sigma}\right),
\]
which implies that $|\epsilon_j| \leq C_0 \sigma \log m$ with probability at least $1 - {e}/{m^{11}}$. This implies that 
\[
\|\vct{\epsilon}\|_\infty \leq C_0  \sigma \log m
\]
with probability at least $1- e/m^{10}$.\\
~\\
Since
\[
\sigma \geq \|\epsilon_j\|_{\Psi_1} = \sup_{p\geq 1} p^{-1}\left(\E|\epsilon_j|^p\right)^{\frac{1}{p}},
\]
we have $\E \epsilon_j^2 \leq (2\sigma)^2$ and $\E \epsilon_j^4 \leq (4\sigma)^4$. Define
\[
X = \frac{1}{m} \sum_{j=1}^m \epsilon_j^2.
\]
Then we have $\E X \leq (2 \sigma)^2$, and
\[
\Var (X) \leq (4 \sigma)^4/m.
\]
By Chebyshev's inequality, we have
\[
\P\left(|X - \E X| \geq t\right) \leq \frac{\Var(X)}{t^2}.
\]
By letting $t=(4\sigma)^2$, we obtain that with probability at least $1 - 1/m$, we have $|X|\leq 20 \sigma^2$.\\
~\\
Similarly, with probability at least $1 - 1/m$, we have $\left|\frac{1}{m} \sum_{j=1}^m\epsilon_j^4\right| \leq C_0 \sigma^4$ for some absolute constant $C_0$.
\end{proof}

\begin{lemma} 
\label{lmm:concentration_noise}
Suppose $\vct{z}_j \in \mathbb{R}^k$, $j=1, \ldots, m$ are IID standard normal random vectors. For fixed $\vct{a} \in \mathbb{R}^m$, with probability at least $1- 2e^{-k}$, we have
\[
\left\|\sum_{j=1}^m a_j \vct{z}_j \vct{z}_j' - \left(\sum_{j=1}^m a_j\right) \mtx{I}_k\right\| \leq C_0 \left(\sqrt{k \|\vct{a}\|_2^2} + k\|\vct{a}\|_\infty\right)
\]
for some absolute constant $C_0$.
\end{lemma}

\begin{proof}
Define 
\[
\mtx{A}: = \sum_{j=1}^m a_j \vct{z}_j \vct{z}_j' - \left(\sum_{j=1}^m a_j\right) \mtx{I}_k.
\]
By Lemma 4 in \cite{vershyninNARMT}, we have 
\[
\left\|\mtx{A}\right\| \leq 2 \sup_{\vct{x} \in \mathcal{N}_{\frac{1}{4}}} |\vct{x}'\mtx{A}\vct{x}|,
\]
where $\mathcal{N}_{\frac{1}{4}}$ is the $1/4$-net of the unit sphere $\mathcal{T}^{k-1}$.\\

For fixed $\vct{x} \in \mathcal{N}_{\frac{1}{4}}$, let $y_j = |\vct{z}_j'\vct{x}|^2 - 1$. Then
\[
\vct{x}'\mtx{A}\vct{x} = \sum_{j=1}^m a_j y_j.
\]
Notice that $y_j$, $j=1, \ldots, m$ are IID sub-exponential variables with $\|y_j\|_{\psi_1} \leq K$ where $K$ is an absolute constant. By Bernstein inequality (see, e.g., Proposition 16 in \cite{vershyninNARMT}), we have with probability at least $1 - 2 \exp(-4k)$,
\[
\left|\sum_{j=1}^m a_j y_j\right| \leq (C_0 /2)\left(\sqrt{k \|\vct{a}\|_2^2} + k\|\vct{a}\|_\infty\right)
\]
for some absolute constant $C_0$.\\

Since $|\mathcal{N}_{\frac{1}{4}}| \leq 9^k$, we know with probability at least $1- 2e^{-k}$, we have
\[
\left\|\mtx{A}\right\| \leq 2 \sup_{\vct{x} \in \mathcal{N}_{\frac{1}{4}}} |\vct{x}'\mtx{A}\vct{x}|\leq C_0 \left(\sqrt{k \|\vct{a}\|_2^2} + k\|\vct{a}\|_\infty\right).
\]
\end{proof}

\bibliographystyle{plainnat}
\bibliography{zm}

\end{document}